\documentclass[a4paper,12pt]{article}

\usepackage{graphicx, amssymb, latexsym, amsfonts, amsmath, lscape, amscd,	amsthm, color, epsfig}
\usepackage{graphicx}

\usepackage[T1]{fontenc}
\usepackage[latin1]{inputenc}

\newtheorem{theorem}{Theorem}[section]
\newtheorem{conjecture}{Conjecture}
\newtheorem{corollary}{Corollary}[section]

\newtheorem{lemma}{Lemma}[section]
\newtheorem{proposition}{Proposition}[section]
\newtheorem{problem}{Problem}

\newtheorem{definition}{Definition}

\newcommand{\lose}{$rs \rightarrow ab $ in $\Delta$, namely $s \rightarrow b$ and $ a \notin N^{++}_D(s)$}

\newcommand{\nd}{$b\in N^{++}_D(f)$ }
\newcommand{\nnd}{$b\in N^{++}_{D'}(f)$ }
\newcommand{\ndn}{$b\in N^{++}_D(f)$}
\newcommand{\nndn}{$b\in N^{++}_{D'}(f)$}
\newcommand{\maxp}{There is a maximal directed path $P=m_0y_0 \rightarrow ..\rightarrow m_iy_i \rightarrow ...\rightarrow m_ky_k$ in $\Delta$ such that }
\newcommand{\same}{The same order $L$ is a local median order of the obtained tournament $T'$, $f$ is also  feed vertex of $L$ and thus $f$ has the SNP in $T'$}
\newcommand{\ra}{\rightarrow}

\date{}

{\large }

\begin{document}
	
	\title{The Second Neighborhood Conjecture for Oriented Graphs Missing $\{C_{4}, \overline{C_{4}}, S_{3},$ chair and co-chair$\}$-Free Graph}

	\maketitle
\begin{center}\author{Darine AL MNINY \footnote[1]{KALMA Laboratory, Department of Mathematics, Faculty of Sciences I, Lebanese University, Beirut - Lebanon. (darinealmniny@gmail.com)}$^{,}$ \footnote[2]{Camille Jordan Institute, Claude Bernard University - Lyon 1, France. (mniny@math.univ-lyon1.fr)}, Salman GHAZAL \footnote[3]{ Department of Mathematics, Faculty of Sciences I, Lebanese University, Beirut - Lebanon. (salman.ghazal@ul.edu.lb)}$^{,}$ \footnote[4]{ Department of Mathematics and Physics, School of Arts and Sciences, Beirut International University, Beirut - Lebanon. (salman.ghazal@liu.edu.lb)}} \end{center}

	\begin{abstract}
	\noindent
	Seymour's Second Neighborhood Conjecture (SNC) asserts that every oriented graph has a vertex whose first out-neighborhood is at most as large as its second out-neighborhood. In this paper, we prove that if $G$ is a graph containing no induced  $C_4$, $\overline{C_4}$, $S_3$, chair and $\overline{chair}$, then every oriented graph missing $G$  satisfies this conjecture. As a consequence,  we deduce that the conjecture holds for every oriented graph missing  a threshold graph, a generalized comb or a star.
	
\end{abstract}

\section{Introduction}
\bigskip

Throughout this paper, all graphs are considered to be simple, that is, there are no loops and no multiple edges. Given a graph $G$, the vertex set and edge set of $G$ are  denoted by $V(G)$ and $E(G)$ respectively. Given an edge $xy$ of $G$,  the vertices $x$ and $y$ are called the endpoints of $xy$ and they are  said to be adjacent.  Two edges of  $G$ are said to be adjacent if they have a common endpoint. The neighborhood of a vertex $v$ in $G$, denoted by $N_G(v)$, is the set of all vertices adjacent to $v$. The  degree $d_G(v)$ of $v$ in $G$  is defined to be $d_G(v):=|N_G(v)|$. Note that we may omit the subscript if the graph is clear from the context. Given two sets of vertices $U$ and $W$ of  $G$, we denote by  $E[U, W]$ the set of all edges in  $G$ that joins a vertex in $U$ to a vertex in $W$. For $A \subseteq V(G)$, $G[A]$ denotes the subgraph of $G$ induced by $A$. If $G[A]$ is an empty graph, then $A$ is called a stable set, that is, there is no edge that joins any two distinct vertices of $A$.  However, if $G[A]$ is a complete graph, then $A$ is called a clique set, that is, any two distinct vertices of $A$ are adjacent.  The complement graph $\overline{G}$ of $G$ is  defined as follows: $V(\overline{G}) = V(G)$ and $xy\in E(\overline{G})$ if and only if $xy\notin E(G)$. A graph $H$ is called forbidden subgraph of  $G$ if $H$ is not (isomorphic to) an induced subgraph of $G$. In this case, we say that $G$ is $H$-free graph. \\

A digraph  is an orientation of a graph so that it contains neither loops nor parallel arcs. However, an oriented graph is a digraph without  digons (directed cycles of length 2).  Given a digraph $D$,  the vertex set and arc set of  $D$ are  denoted by $V(D)$ and $E(D)$ respectively. For  $(x,y)\in E(D)$ with $x, y\in V(D)$,  we say that $y$ is an out-neighbor of $x$, $x$ is an in-neighbor of $y$ and $x$ and $y$ are adjacent.  The (first) out-neighborhood (resp. in-neighborhood) $N^{+}_D(v)$ (resp. $N^{-}_D(v)$) of a vertex $v$ in $D$  is the set of all out-neighbors (resp. in-neighbors) of $v$. Moreover, the second out-neighborhood  $N^{++}_D(v)$  of $v$ in $D$ is the set of vertices that are at distance 2 from $v$, that is,  $N^{++}_D(v):=\{ x \in V(D)- N^{+}_D(v);  $ $ \exists$ $  y \in  N^{+}_D(v) |$
$ (y,x) \in  E(D)\}$. The out-degree, the in-degree and the second out-degree of $v$ in $D$ are defined as follows: $d^{+}_D(v):=|N^{+}_D(v)|$, $d^{-}_D(v):=|N^{-}_D(v)|$ and $d^{++}_D(v):=|N^{++}_D(v)|$, respectively. Note that  we omit the subscript if the digraph is clear from the context. For short, we write $x\rightarrow y$ if the arc $(x,y)\in E(D)$. Also, we write $x_1\ra x_2\ra ...\ra x_n$, if $x_i\ra x_{i+1}$ for every $1 \leqslant i \leqslant n-1$. \\

 Let $D$ be an oriented graph and let $v \in V(D)$, we say that $v$ has the second neighborhood property SNP if $d^{+}(v)\leq d^{++}(v)$. In 1990, P. Seymour  \cite{dean} conjectured the following:

 \begin{conjecture}
 	 Every oriented graph has a vertex satisfying the SNP.
 \end{conjecture}

The above conjecture is called "The Second Neighborhood Conjecture", and it is abbreviated by  "SNC". The SNC on tournaments is called Dean's conjecture, where tournaments are orientations of complete graphs. In 1996, Fisher \cite{fisher} proved Dean's Conjecture. In 2000,
a shorter proof of Dean's conjecture was given by Havet and Thomass\'{e} \cite{m.o.} using a tool called the  median order. In 2007, Fidler and Yuster \cite{fidler} proved the SNC for tournaments missing a matching, using local median orders and dependency digraphs. In 2012, Ghazal  \cite{a} proved the weighted version of SNC for tournaments missing a generalized star. Then  in 2013  Ghazal \cite{contrib} proved the SNC for tournaments missing a comb, cycle of length 4 or 5. In 2015, Ghazal \cite{ghazal3} refined the result of \cite{fidler} and he showed in particular  that  every tournament missing a matching has  a certain "feed vertex" satisfying the SNP. \\

In this paper, we prove the SNC for any oriented graph  missing  a graph $G$, where  $G$   contains no  $C_4$, $\overline{C_4}$, $S_3$, chair and $\overline{chair}$ as induced subgraphs. This generalizes the results of \cite{a} and \cite{contrib}.

\section{Definitions and Preliminaries}
\bigskip

 A chair is a graph $G$  whose vertex set is  $V(G)= \{x,y,z,t,v\}$ and whose edge set is $E(G)= \{xy,yz,zt, zv\}$. The co-chair or $\overline{chair}$ is defined to be  the complement of a chair. We denote by $C_n=v_1v_2...v_nv_1$ the cycle on $n$ vertices, by $P_n=v_1v_2...v_n$ the path on $n$ vertices and  by $S_3$  the graph  on 6 vertices  indicated in Figure \ref{s1}.

\begin{figure}[h]
	\centering
	\begin{minipage}[b]{0.3\textwidth}
		\includegraphics[width=\textwidth]{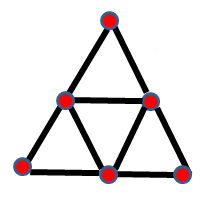}
		\caption{$S_3$}
		\label{s1}
	\end{minipage}
     \hfill
\begin{minipage}[b]{0.3\textwidth}
	\includegraphics[width=\textwidth]{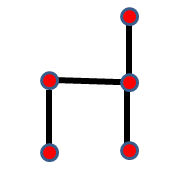}
	\caption{Chair}
\end{minipage}
	\hfill
	\begin{minipage}[b]{0.3\textwidth}
		\includegraphics[width=\textwidth]{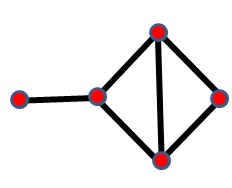}
		\caption{Co-chair}
	\end{minipage}

\end{figure}

A graph $G$ is a called a split graph if its vertex set is the disjoint union of a stable set $S$ and a clique set $K$. In this case, we write $G$ is an $\{S$, $K\}$-split graph. For an $\{S$, $K\}$-split graph $G$, if $sx\in E(G)$  $\forall$ $ s\in S$ and $\forall x\in K$, then $G$ is called a complete split graph. Otherwise if $E[S,K]$ forms a perfect matching of $G$, then $G$ is called a perfect split graph. \\

In \cite{thresholdch} and \cite{threshold}, the notion of a threshold graph is introduced as follows:

\begin{definition}
	A graph $G$ is called a threshold graph if:
	\begin{description}
		\item[1) ] $V(G)=\displaystyle\bigcup_{i=1}^{n+1}(X_i\cup A_{i-1})$, where the $A_i$'s and the $X_i$'s are pairwisely disjoint sets.
		\item[2) ] $K:=\displaystyle\bigcup_{i=1}^{n+1}X_i$ is a clique and the $X_i$'s are nonempty, except possibly $X_{n+1}$.
		\item[3) ] $S:=\displaystyle\bigcup_{i=0}^{n}A_{i}$ is a stable set and the $A_i$'s are nonempty, except possibly $A_0$.
		\item[4) ] $\forall$ $ 1\leq j\leq i\leq n$, $G[A_i\cup X_j]$ is a complete split graph.
		\item[5)] The only edges of $G$ are the edges of the subgraphs mentioned above.
	\end{description}
In this case, $G$ is called an $\{S,$ $K\}$-threshold graph.\\
\end{definition}

On the structure of threshold graphs, Hammer and Chv\`{a}tal noticed the following:

\begin{theorem}(Hammer and Chv\`{a}tal \cite{thresholdch}, \cite{threshold})\label{thre}
$G$ is a threshold graph if and only if $C_4$, $\overline{C}_4$ and $P_4$ are forbidden subgraphs of $G$.\\
\end{theorem}

As a generalization of threshold graphs, Ghazal introduced the notion of  generalized combs and he  studied their structures and properties as follows:

\begin{definition}(Ghazal \cite{combch}) \label{gcdef}
	A graph $G$ is called a generalized comb if:
	\begin{description}
		\item[1)] $V(G)$ is disjoint union of sets $A_0,...,A_n, M_1,...,M_l,X_1,....,X_{n+1}, Y_2,...,Y_{l+2}$ with $Y_1=X_1$. These sets are called the sets of the generalized  comb $G$.
		\item[2)] $S:=A\cup M$ is a stable set, where $M=\displaystyle\bigcup_{i=1}^{l}M_{i}$ and $A=\displaystyle\bigcup_{i=0}^{n}A_{i}$.
		\item[3)] $K:=X\cup Y$ is a clique, where $X=\displaystyle\bigcup_{i=1}^{n+1}X_{i}$ and $Y=\displaystyle\bigcup_{i=1}^{l+2}Y_{i}$.
		\item[4)] $\forall$ $ 1\leq j\leq i\leq n$, $G[A_i\cup X_j]$ is a complete split graph.
		\item[5)]$G[A\cup Y]$ is a complete split graph.
		\item[6)]$\forall$ $ 1\leq i\leq l$, $G[Y_i\cup M_i]$ is a perfect split graph or $M_i=\phi$.
		\item[7)] $\forall$ $ 1\leq i <j \leq l+1$, $G[Y_j\cup M_i]$ is a complete split graph.
		\item[8)] $X_{n+1}, Y_{l+2}, Y_{l+1}$ and $A_0$ are the only possibly empty sets among the $X_i's$, $Y_i's$ and $A_i's$.
		\item[9)] The only edges of $G$ are the edges of the subgraphs mentioned above.
	\end{description}
In this case, we say that $G$ is an $\{S,$ $K\}$-generalized comb.
\end{definition}

\begin{theorem}(Ghazal \cite{combch}) \label{generalisedcom}
$G$ is a generalized comb if and only if $C_4$, $\overline{C}_4$, $C_5$, $S_3$, chair and co-chair are forbidden subgraphs of $G$.
\end{theorem}

\begin{corollary}(Ghazal \cite{combch})
Every threshold graph is a generalized comb.
\end{corollary}

\begin{proposition} \label{gs in gc}
	Let $G$ be a generalized comb defined as in Definition \ref{gcdef}. Then $G'=G-\bigcup_{1\leq i\leq l}E[Y_i, M_i]$ is a threshold graph.
\end{proposition}

\begin{proof}
	
	It is clear that $G'$ contains no induced $C_4$, $\overline{C}_4$ or $P_4$. Consequently, Theorem \ref{thre} implies that $G'$ is a threshold graph.  \end{proof}

On the structure of graphs containing no 	$C_4$, $\overline{C_4}$, $S_3$, chair and co-chair as induced subgraphs, Ghazal remarked the following:

\begin{theorem}(Ghazal \cite{combch}) \label{allowedC5}
	$C_4$, $\overline{C_4}$, $S_3$, chair and co-chair are forbidden subgraphs of a graph $G$ if and only if $V(G)$ is disjoint union of three sets $S$, $K$ and $C$ such that:
	
	\begin{description}
		\item[1)] $G[S\cup K]$ is an $\{S, K\}$-generalized comb.
		\item[2)] $G[C]$ is empty or isomorphic to the cycle $C_5$.
		\item[3)] Every vertex in $C$ is adjacent to every vertex in $K$ but to no vertex in $S$.
	\end{description}
	
\end{theorem}
From now on, if $G=C_5$, we set $G= xyzuvx$. If $G$ is an $\{S, K\}$-generalized comb, we follow the same notations as in Definition \ref{gcdef}. Moreover, if $G$ is a  $\{C_4$, $\overline{C_4}$, $S_3$, chair and co-chair$\}$-free graph, we use the notations in Theorem \ref{allowedC5}. Note that if $G$  is defined as in Theorem \ref{allowedC5} and $G[C]$ is empty, then $G$ is a generalized comb.
\section{Characterization Using Dependency Digraphs}
\bigskip
Let $D$ be an oriented graph. For two vertices $x$ and $y$ of $D$, we say that $xy$ is a missing edge of $D$ if $(x,y) \notin E(D)$ and $(y,x)\notin E(D)$. A vertex $v$ of $D$ is called a whole vertex if it is not incident to any missing edge, i.e., $N^{+}(v)\cup N^{-}(v)=V(D)-\{v\}$. Otherwise, we say that $v$ is a non-whole vertex. The missing graph $G$ of $D$ is defined to be the graph formed by the missing edges of $D$, formally, $G$ is the graph whose edge set  is the set of all the missing edges of $D$ and whose vertex set is the set of the non-whole vertices. In this case, we say that $D$ is missing $G$. Given two missing edges $x_1y_1$ and $x_2y_2$, we say that $x_1y_1$ loses to $x_2y_2$ if: $x_1\rightarrow x_2$ and $y_2\notin N^{+}(x_1)\cup N^{++}(x_1)$, $y_1\rightarrow y_2$ and $x_2\notin N^{+}(y_1)\cup N^{++}(y_1)$.\\

 The dependency digraph $\Delta_{D}$ (or simply $\Delta$) of $D$  is defined to be the digraph whose  vertex set consists of all the missing edges of $D$, and whose arc set contains the arc $(ab,cd)$ if and only if the missing edge $ab$ loses to the missing edge $cd$. Note that $\Delta$ may contain digons. These digraphs were used in \cite{fidler, contrib} to prove SNC for some oriented graphs.\\

\begin{definition} (Ghazal \cite{a})	A missing edge $ab$ is called  good  if one of the following holds:
	
	\begin{description}
		\item[(i)] $(\forall v \in V\backslash\{a,b\})[(v\rightarrow a)\Rightarrow(b\in N^{+}(v)\cup N^{++}(v))];$
		\item[(ii)]  $(\forall v \in V\backslash\{a,b\})[(v\rightarrow b)\Rightarrow(a\in N^{+}(v)\cup N^{++}(v))]$.
	\end{description}
If $ab$ satisfies $(i)$ we say that $(a,b)$ is a convenient orientation of $ab$. Else, we say that $(b,a)$ is a convenient orientation of $ab$.\\
\end{definition}

\begin{lemma} (Ghazal \cite{contrib})\label{goodmissinedgelemma}
	Let $D$ be an oriented graph and let $\Delta$ denote its dependency digraph. A missing edge $ab$ is good if and only if its in-degree in $\Delta$ is zero.
\end{lemma}

In \cite{a}, threshold graphs are characterized using dependency digraphs as follows:

\begin{theorem}(Ghazal \cite{a}) \label{threshold}
	Let $G$ be a graph. The following statements are equivalent:
	\begin{description}
		\item[i)] $G$ is a threshold graph;
		\item[ii)] Every missing edge of every oriented graph missing $G$ is good;
		\item[iii)] The dependency digraph of every oriented graph missing $G$ is empty.\\
	\end{description}
\end{theorem}
\begin{problem}
	Let $\vec{\mathcal{P}}$ be the family of all digraphs consisting of vertex disjoint directed paths and let $\mathcal{F}(\vec{\mathcal{P}})=\{G$ is a graph; $\forall$ $ D$ missing $G$, $\Delta_{D}\in \vec{\mathcal{P}}\}$ .  Characterize $\mathcal{F}(\vec{\mathcal{P}})$.
\end{problem}

\begin{proposition}\label{inducedsub}
	$G \in \mathcal{F}(\vec{\mathcal{P}})$ if and only if $G' \in \mathcal{F}(\vec{\mathcal{P}})$, for every $G'$ induced subgraph of $G$.
\end{proposition}

\begin{proof}
	
	 \textit{Sufficient Condition.} Simply take $G'=G$.\\
	
	  \noindent\textit{Necessary Condition.} Assume first that $G'=G-v$ for some $v\in V(G)$. Let $D'$ be any oriented graph missing $G'$. Let $\alpha$ and $\beta$ be 2 distinct extra vertices neither in $D'$ nor in $G$.  Define $D$ as follows: The missing graph of $D$ is $G$, $V(D)=V(D')\cup \{v,\alpha,\beta\}$, that is, $D-\{v,\alpha,\beta\}=D'$. The arcs $(\alpha, v)$, $(v, \beta)$ and $(\alpha, \beta)$ are in $D$. For every $x\in V(D')$, if $xv\notin E(G)$, then $(x,v)\in E(D')$. Finally, for every $x\in V(D')$, the arcs $(x,\alpha)$ and $(\beta, x)$ are in $D$. Then the addition of $v, \alpha$ and $\beta$  to $D'$ in this way  neither affects the losing relations between the missing edges of $D'$ nor creates new ones. Hence, $\Delta_{D}$ is equal to $\Delta_{D'}$ plus isolated vertices (these isolated vertices are the edges of $G$ incident to $v$). Since $D$ is missing $G$, then $\Delta_{D}\in \vec{\mathcal{P}} $.  Whence,  $\Delta_{D'}\in \vec{\mathcal{P}}$. Thus $G'\in \mathcal{F}(\vec{\mathcal{P}})$. Now, the proof follows by induction on the number of vertices removed from $G$ to obtain the induced subgraph.

\end{proof}

It is  easy to observe the following:
\begin{proposition}\label{notinFP}
	$\overline{C}_4$, chair and co-chair are not in $\mathcal{F}(\vec{\mathcal{P}})$.
\end{proposition}

\begin{proof}
	Let $D$ be the oriented graph with vertex set $V(D)=\{a, b, c, d\}$ and arc set $E(D)=\{(a,c), (b,d),  (d,a), (c,b)\}$. Then $D$ is missing $\overline{C}_4$, $ab$ loses to $cd$ and $cd$ loses to $ba$. Thus $\Delta_{D} \notin \vec{\mathcal{P}}$.\\
	
	\noindent Let $D'$ be the oriented graph with  vertex set $V(D')=\{a, b, c, d, x\}$ and arc set $E(D')=\{(a,d), (b,c), (c,a), (b,x), (x,a), (x,c)\}$. Then $D'$ is missing a chair, $ab$ loses to both $dc$ and $dx$. Thus $\Delta_{D'}\notin \vec{\mathcal{P}}$.\\
	
	\noindent Let $D''$ be the oriented graph  with  vertex set $V(D'')=\{a, b, c, d, x\}$ and arc set $E(D'')=\{(a,c), (b,d), (d,a), (a,x)\}$. Then $D''$ is missing a co-chair, $ab$ loses to both $dc$ and $dx$. Thus $\Delta_{D''}\notin \vec{\mathcal{P}}$.\end{proof}

\begin{proposition}\label{c5}(Ghazal \cite{contrib})
	 $C_5 \in \mathcal{F}(\vec{\mathcal{P}})$.
\end{proposition}
\begin{proof}
 Let $D$ be an oriented graph missing $C_5$ and let $\Delta$ denote its dependency digraph. Then we may check by cases that one of the following occurs up to isomorphism:

\begin{description}
	\item[(i)] $\Delta$ has no arcs.
	\item[(ii)] $\Delta$ has exactly one arc, say $uv\rightarrow xy$.
	\item[(iii)] $\Delta$ has exactly two arcs, say $uv \rightarrow xy$ and $xv \rightarrow yz$.
	\item[(iv)] $\Delta$ has exactly two arcs, say $uv\rightarrow xy \rightarrow zu$.
	\item[(v)] $\Delta$ has exactly three arcs, say $uv \rightarrow xy \rightarrow zu \rightarrow vx$.
	\item[(vi)] $\Delta$ has exactly three arcs, say $uv \rightarrow xy \rightarrow zu$ and $xv \rightarrow zy$.
\end{description}
In all the cases above, $\Delta$ consists of vertex disjoint directed paths of lengths at most 3, that is,  $\Delta \in \vec{\mathcal{P}} $.  Thus the result holds.\end{proof}

\begin{theorem}\label{comb}
	 Let $G$ be a graph having no induced $C_4$, $C_5$, nor $S_3$. Then $G\in \mathcal{F}(\vec{\mathcal{P}})$ if and only if $G$ is a generalized comb.
\end{theorem}

\begin{proof}
\textit{Necessary Condition}. Since $G\in \mathcal{F}(\vec{\mathcal{P}})$, then Proposition \ref{inducedsub} together with Proposition \ref{notinFP} imply that $\overline{C}_4$, chair  and co-chair are not induced subgraphs of $G$. But  $C_4$, $C_5$, $S_3$ are not also induced subgraphs of $G$, thus due to Theorem \ref{generalisedcom}  $G$ is a generalized comb.\\	

\noindent \textit{Sufficient Condition}. Let $D$ be an oriented graph missing a generalized comb $G$ and let $\Delta$ denote its dependency digraph. Using the definition of $G$, each possible losing relation can occur between two edges in $E[Y_t, M_t]$, for some $t$. For $i=1,2,3$, suppose that $a_ix_i\in E[Y_t, M_t]$ with $a_i\in M_t$ and $x_i\in Y_t$. Assume that $a_1x_1$ loses to the two others. Then we have $a_1\rightarrow x_3$, $x_1\rightarrow a_2$, $a_2\notin N^{+}(a_1) \cup N^{++}(a_1)$
and $x_3\notin N^{+}(x_1) \cup N^{++}(x_1)$. By definition of $G$, the only edge of $G[Y_t\cup M_t]$ incident to $a_2$ is $a_2x_2$. Thus  $a_2x_3$ is not a missing edge and so either $a_2\rightarrow x_3$ or
$x_3 \ra a_2$. Whence, either $x_3\in N^{+}(x_1) \cup N^{++}(x_1) $ or $a_2\in N^{+}(a_1) \cup N^{++}(a_1) $, a contradiction. Thus the maximum out-degree in $\Delta$ is 1. Similarly, we can prove that the maximum in-degree in $\Delta$ is 1. This implies that $\Delta$ is composed of directed cycles and paths only.\\

Assume now that $\Delta$ contains a directed cycle $a_1b_1 \rightarrow...\rightarrow a_nb_n \rightarrow a_1b_1$, with $a_i \in M_t$ and $b_i \in Y_t$, for some $t$. Then by the losing relations we must have $a_{i+1} \rightarrow a_i $ $ \forall$ $ i<n$ and $a_1 \rightarrow a_n$ in $D$. We will show now by induction on $i$ that $\forall$ $ 1\leq i<n $, $a_i \rightarrow a_n$. It is true for $i=1$. Assume it is true for $i-1$. Then $a_{i-1} \rightarrow a_n$. Since $a_{i-1}b_{i-1}$ loses to $a_ib_i$, then $a_i\notin N^{++}(a_{i-1})$. But $a_ia_n$ is not a missing edge of $D$, then we must have $a_i \rightarrow a_n$, since otherwise $a_{i-1} \rightarrow a_n \rightarrow a_i$ in $D$,  a contradiction. This proves  that $\forall$ $ 1\leq i<n $, $a_i \rightarrow a_n$. In particular,  $a_{n-1} \rightarrow a_n$, a contradiction. Thus $\Delta$ has no directed cycles. This shows that $G\in\mathcal{F}(\vec{\mathcal{P}})$.

\end{proof}

As a consequence, we deduce the following on the  characterization of the graphs of our interest:

\begin{corollary}\label{ourgraph}
Given a graph $G$, the following statements are equivalent:
\begin{description}
	\item[i)] $C_4$ and $S_3$ are forbidden subgraphs of $G$ and $G\in \mathcal{F}(\vec{\mathcal{P}})$;
	\item[ii)] $C_4$, $\overline{C_4}$, $S_3$, chair and co-chair are forbidden subgraphs of a graph $G$;
	\item[iii)] $V(G)$ is disjoint union of three sets $S$, $K$ and $C$ such that: \begin{description}
		\item[1)] $G[S\cup K]$ is an $\{S, K\}$-generalized comb;
		\item[2)] $G[C]$ is empty or isomorphic to the cycle $C_5$;
		\item[3)] every vertex in $C$ is adjacent to every vertex in $K$ but to no vertex in $S$.
	\end{description}
\end{description}

\end{corollary}

\begin{proof}
Due to Theorem \ref{allowedC5}, ii) and iii) are equivalent.  However, ii) follows from i) due to Proposition \ref{inducedsub}  and Proposition \ref{notinFP}. Now assume that iii) holds and let $D$ be an oriented graph missing $G$.    Observe that every edge in $E(G)-E(C)$ is incident to a vertex in $K$. This implies that there is no losing relation between an edge in $E(C)$ and an edge in  $E(G)-E(C)$, since otherwise there is an edge $ab$ with $a \in C$, $b \in K$ and  $ab \notin E(G)$. This contradicts  the fact that every vertex in $C$ is adjacent in $G$ to every vertex in $K$. In the same way, we can prove that there is no losing relations  between an edge in $E(G[S \cup K])$ and an edge in $E[K,C]$, or between two edges in $E[K,C]$. Thus the only  possible losing relations hold  either between two edges in  $G[S\cup K]$ or between two edges in $G[C]$.  However, $G[S\cup K]$ is a generalized comb, then by Theorem \ref{comb} it is in $\mathcal{F}(\vec{\mathcal{P}})$. Moreover, $G[C]$ is empty or isomorphic to the cycle $C_5$, whence by Proposition \ref{c5} it is in $\mathcal{F}(\vec{\mathcal{P}})$. Therefore, $G\in \mathcal{F}(\vec{\mathcal{P}})$ and so i) holds.
\end{proof}
 We will use the following lemma frequently in the proof of our main theorem:

 \begin{lemma}\label{snlemma}
 	Suppose that $rs$ loses to $ab$ and $f \rightarrow a$ in an oriented graph $D$, namely with $s \rightarrow b$. If $fs$ is not a missing edge of $D$, then $f \rightarrow s \rightarrow b$ in $D$ and thus $b\in N^+(f)\cup N^{++}(f)$.
 \end{lemma}

 \begin{proof}
 	Since $fs$ is not a missing edge, then either $f \rightarrow s$ or $s \rightarrow f$ in $D$. If $s \rightarrow f$ in $D$, then $s \rightarrow f \rightarrow a$ in $D$ and thus $a\in N^{++}(s)$, which contradicts the fact that $rs$ loses to $ab$. Thus $f \rightarrow s$. Whence, the result follows.\end{proof}

\section{Main Theorem}
\bigskip

Let $L=v_1v_2...v_n$ be an ordering of the vertices of a digraph $D$.   An arc $(v_i,v_j)\in E(D)$ is called forward with respect to $L$ if $i<j$. Otherwise, it is called backward with respect to $L$. $L$ is called a local median order of $D$ if it maximizes the set of  forward arcs of $D$ w.r.t. $L$, that is, the set $\{(v_i,v_j)\in E(D);$ $i<j \}$. In this case, $L$ satisfies the feedback property: For all $1\leq i\leq j\leq n:$
$$ d^{+}_{]i,j]}(v_i)  \geq  d^{-}_{]i,j]}(v_i); $$
$$and$$
$$ d^{-}_{[i,j[}(v_j) \geq  d^{+}_{[i,j[}(v_j), $$
where $]i,j]:=D[\{v_{i+1}, v_{i+2},  ...,v_j\}]$ and  $[i,j[:=D[\{v_{i}, v_{i+1}, ...,v_{j-1}\}]$.\\

\noindent In this case, the last vertex $v_n$ is called a feed vertex.\\

 We will need the following proposition:

\begin{proposition}
	Suppose that $L=v_1v_2...v_n$ is a local median order of a digraph $D$ and $e=(v_j,v_i)\in E(D)$ with $i<j$.
	Then $L$ is a local median order of the digraph $D'$ obtained from $D$ by reversing the orientation
	of $e$.
\end{proposition}

We will use the following theorem:

\begin{theorem}(Havet et al. \cite{m.o.}) \label{feed}
	Every feed vertex of a tournament has the SNP.
\end{theorem}

Now we are ready to prove our main theorem:

\begin{theorem} Let $D$ be an oriented graph missing a 	$\{C_4$, $\overline{C_4}$, $S_3$, chair and co-chair$\}$-free graph $G$. Then $D$ satisfies the SNC.\end{theorem}
\begin{proof}
Let $\Delta$ denote the dependency digraph of $D$, and let $\Delta[E(C)]$ denote the subdigraph of $ \Delta $ induced by the set of vertices that correspond to the edges pf $C$.  Corollary \ref{ourgraph} follows that $\Delta$ consists of disjoint directed paths only and each of its arcs occurs only between two edges in the same set $E[Y_{j},M_{j}]$ for some $j$ or between two edges of $C$. Let $P= m_{0}y_{0} \rightarrow ... \rightarrow  m_{i}y_{i} \rightarrow ... \rightarrow m_{k}y_{k}$ be a maximal directed path in $\Delta$, with $ m_{i} \in M_{j}$ and $ y_{i} \in Y_{j}$. Due to the maximality of $P$ and due to Lemma \ref{goodmissinedgelemma}, $m_{0}y_{0}$ is a good missing edge and so it has a convenient orientation. If $ (m_{0},y_{0})$ is a convenient orientation, we add the arcs $(m_{2i},y_{2i})$ and the arcs $(y_{2i+1}, m_{2i+1})$ to $D$. Else, we add the arcs $(y_{2i},  m_{2i})$ and the arcs  $(m_{2i+1}, y_{2i+1})$. We do this for every maximal directed path in $\Delta$ whose vertices are edges in  $E[Y_{j},M_{j}]$. We have many cases.\\

\noindent \textbf{\underline{Case $ i $:}} \textit{$ G[C]$ is empty or $ \Delta[E(C)] $ is empty.}\\

\noindent In this case, the obtained oriented graph $D'$ is missing $G'=G-\cup E[Y_j, M_j]$ which is a threshold graph by Proposition \ref{gs in gc}. We assign to every missing edge of $D'$ (which is good by Theorem \ref{threshold}) a convenient orientation and we add it to $D'$. The obtained oriented graph $T$ is a tournament. Let $L$ be a local median order of $T$ and let $f$ denote its feed vertex. Then by Theorem \ref{feed} $f$ has the SNP in $T$. Reorient all the missing edges incident to $f$ towards $f$ except those whose out-degree in $\Delta$ is not zero. \same . We will prove that $f$ has the SNP in $D$ also. For this aim, we consider many cases.\\

\textbf{\underline{Case $  1 $:}} \textit{$f$ is a whole vertex.} Clearly,  $f$ gains no new out-neighbor. We will prove that $f$ gains no new second out-neighbor. Assume that $f \rightarrow a \rightarrow b \rightarrow f$ in $T'$. Then $f \rightarrow a$ and $b \rightarrow f$ in $D$.  If $a \rightarrow b $ in $E(D)$ or $(a,b)$ is a convenient orientation w.r.t. $D$, then $b \in N^{++}_D(f)$. If $a \rightarrow b$ in $E(D')-E(D)$ and $(a,b)$ is not a convenient orientation w.r.t. $D$, then there is  \lose. But $f \rightarrow a$ in $D$ and $fs$ is a non-missing edge of $D$, then by Lemma \ref{snlemma} \ndn. If $a \rightarrow b$ in $E(T')-E(D')$, then $(a,b)$ is a convenient orientation w.r.t $D'$. Hence \nnd  and so there is $a'$ such that $f \rightarrow a' \rightarrow b$ in $D'$. Since $f \rightarrow a'$ in $D'$ and $f$ is whole vertex, then $f \rightarrow a'$ in $D$. But this is already treated above, thus \ndn.\\

\textbf{\underline{Case $ 2 $:}} \textit{$\exists$ $ 1\leq t\leq l$ such that $f\in M_t$. \maxp $f=m_i$}.\\

\textit{\underline{Case $  2.1 $:} Assume $(y_i,m_i)\in E(D')$}. Clearly, $f$ gains no new first out-neighbor. We claim that $f$ gains no new second out-neighbor. Assume $m_i\ra a\ra b\ra m_i$ in $T'$. Then $(m_i, a)\in E(D)$ and $(a,b)\in E(T)$.\\

\textit{Subcase $ a $}: If $(a,b) \in E(D)$, then clearly \ndn.\\

\textit{Subcase $ b $}: If $(a,b) \in E(D')-E(D)$, then either $(a,b)$ is a convenient orientation w.r.t. $D$ and hence \nd or there is \lose. There is $j$ such that $rs, ab \in E[Y_j, M_j]$. Assume $m_i=r$, then $y_i=s$, $a=y_{i+1}$ and $b=m_{i+1}$. Since $(y_i,m_i)\in E(D')$, then $(m_{i+1}, y_{i+1})\in E(D')$, that is, $(b,a)\in E(D')$, a contradiction. So $m_i\neq r$. Assume now that $s=m_i$. Then $a=m_{i+1}$. However $(m_{i+1},m_i)\in E(D)$, then $(a,m_i)\in E(D)$, a contradiction. So $s\neq m_i$. Now we prove that $m_is$ is not a missing edge of $D$. If $b \in Y_j$, then $s\in M_j$ and thus $m_is$ is not a missing edge. Else $ b \in M_j$, whence $a \in Y_j$ and $s\in Y_j$. Since $m_ia$ is not a missing edge, then by definition of $G$, $fs=m_is$ is also not missing edge. But $f\ra a$ in $E(D)$ and $a \notin N^{++}_D(s)$, then due to Lemma \ref{snlemma} we get $b \in N^{++}_D(f)$.\\

\textit{Subcase $ c $}: If $(a,b)\in E(T)-E(D')$. Then $(a,b)$ is a convenient orientation w.r.t $D'$. But $f\ra a$ in $E(D)$ and so in $E(D')$, then $b\in N^{++}_{D'}(f)$. But this is already treated above in Subcase $i.2.1.a $ and Subcase $i.2.1.b$.\\

\textit{\underline{Case $ 2.2 $:} Assume $(m_i,y_i)\in E(D')$}. Here there are two cases to be consider.\\

\textit{Case $ 2.2.1 $: Assume $i=k$, that is,  $f=m_k$}. Clearly, $f$ gains no new out-neighbor. We will prove that $f$ gains no new second out-neighbor. Suppose $f \ra a \ra b \ra f$ in $T'$. Then $(f,a)\in E(D)$ and $(a,b)\in E(T)$.\\

\textit{Subcase $ a $}: If $(a,b)\in E(D)$, then clearly \ndn.\\

\textit{Subcase $ b $}: If $(a,b)\in E(D')-E(D)$, then either $(a,b)$ is a convenient orientation w.r.t. $D$ and hence \nd or there is \lose . There is $j$ such that $rs, uv\in E[Y_j, M_j]$. Since $f=m_k$, we have $r\neq m_k$ and $s\neq m_k$. If $b \in Y_j$, then $s\in M_j$. Then $m_ks$ is not a missing edge. Else $b \in M_j$. Whence, $a \in Y_j$ and $s\in Y_j$. Since $m_ka$ is not a missing edge, then by definition of $G$, $fs=m_ks$ is also not missing edge. But $f\ra a$ in $D$ and $a \notin N^{++}_D(s)$,  then by Lemma \ref{snlemma} we get $b \in N^{++}_D(f)$.\\

\textit{Subcase $ c $}: If $(a,b)\in E(T)-E(D')$, then  $(a,b)$ is a convenient orientation w.r.t. $D'$ and so  \nndn . Then there is a vertex $a'$ such that $m_k\ra a'\ra b$ in $D'$. Since $(m_k,a)\in E(D)$, then $a \neq y_k$ and $\forall$ $ j>t, a \notin Y_j$.\\ Assume $a'=y_k$. Then $(y_k,b)\in E(D')$ and $b \neq m_k$. Thus $(y_k,b)\in E(D)$. This means that  $b \notin A\cup X\cup Y \cup C$. Then either $b$ is a whole vertex or $b\in M$. If $b$ is whole, then $ab$ is not a missing edge, a contradiction. So $b \in M$. Whence, $\exists$ $ \alpha$ such that $b \in M_{\alpha}$.  If $\alpha < t$, then by definition of $G$, $y_kb \in E(G)$, that is,  $y_kb$ is a missing edge, a contradiction. Thus $\alpha\geq t$. Since $b \in M_{\alpha}$ with $\alpha\geq t$ and  $ab$ is  a missing edge of $D'$, then by definition of $G$, $a \in Y_j$ for some $j > \alpha$. Thus $a \in Y_j$ for some $j > t$, a contradiction.\\ So $a' \neq y_k$. Then $(m_k, a') \in E(D)$. But this is treated in Subcase $i.2.2.1.a $ and Subcase $i.2.2.1.b$.\\

\textit{Case $ 2.2.2 $: Assume $i<k$}. Then $f$ gains only $y_i$ as an out-neighbor. We will prove that $f$ gains only $m_{i+1}$ as a second out-neighbor.\\

\textit{Subcase $ a $}: Suppose that $m_i\ra y_i\ra b$ in $T'$ such that $b \neq m_{i+1}$. Then $(y_i, b)\notin E(D')-E(D)$ and $(y_i,b)\in E(T)$.\\

\textit{Subcase $ a.1 $}: If $(y_i,b)\in E(D)$. Since $y_i \in Y_t$, then $y_{i+1} \in Y_t$. Since $y_ib$ is not a missing edge, then by definition of $G$, $y_{i+1}b$ is not a missing edge. Since $y_i\ra b$ in $E(D)$ and $y_{i+1}\notin N^{++}_D(y_i)$, then we must have $y_{i+1} \ra b$. Then $m_i\ra y_{i+1}\ra b$ in $D$.\\

\textit{Subcase $ a.2 $}: If $(y_i, b)\in E(T)-E(D')$. Then $(y_i,b)$ is a convenient orientation w.r.t. $D'$. Then $b \in N^{++}_{D'}(m_i)$, that is, there is vertex $a$ such that $m_i \ra a \ra b$ in $D'$.\\

\textit{Subcase $ a.2.1 $}: If  $a=y_i$, then by definition of $G$, $(y_i,b)\in E(D)$, a contradiction to the fact that $y_ib$ is a missing edge of $D$.\\

\textit{Subcase $  a.2.2 $}: If $a \neq y_i$, then $(m_i,a)$ in $E(D)$. If $(a,b)\in E(D)$, then clearly \ndn. Else $(a,b)\in E(D')- E(D)$, then either $(a,b)$ is a convenient orientation w.r.t. $D$ and hence \nd or there is \lose. Thus  $\exists$ $ j$ such that $rs, ab \in E[Y_j,M_j]$. Assume $r=m_i$. Then $y_i=s$, $a=y_{i+1}$ and $b=m_{i+1}$, a contradiction to the fact that  $b \neq m_{i+1}$. So $r \neq m_i$. Assume $s=m_i$. Then $a=m_{i+1}$. However $(m_{i+1},m_i)\in E(D)$, then $(a,m_i)\in E(D)$, a contradiction. So $s\neq m_i$. Now we will prove that $m_is$ is not a missing edge. If $b \in Y_j$, then $s\in M_j$ and thus $m_is$ is not a missing edge. Else $b \in M_j$. Whence, $a \in Y_j$ and $s\in Y_j$. Since $m_ia$ is not a missing edge, then by definition of $G$, $fs=m_is$ is also not missing edge. But $f\ra a$ in $D$ and $a \notin N^{++}_D(s)$, then by Lemma \ref{snlemma} we get $b \in N^{++}_D(f)$.\\

\textit{Subcase $ b $}: Suppose $m_i\ra a \ra b$ in $T'$ with $a \neq y_i$ and $b \neq m_{i+1}$. Then $(m_i,a)\in E(D)$ and $(a,b)\in E(T)$.\\

\textit{Subcase $  b.1 $}: If $(a,b)\in E(D')$. This is the same as Subcase $i.2.2.2.a.2.2$.\\

\textit{Subcase $ b.2 $}: If $(a,b) \in E(T)-E(D')$, then $(a,b)$ is a convenient orientation w.r.t. $D'$ and thus \nndn. This means that there is a vertex $a'$ such that $m_i\ra a'\ra b$ in $D'$.  If  $a'=y_i$, then  $(y_i,b)\in E(D)$. This case is already treated in Subcase $i.2.2.2.a.1$. Else, we proceed in the same method  as in Subcase $i.2.2.2.a.2.2$.\\

\textbf{\underline{Case $ 3 $:}} \textit{$\exists$ $ 1\leq t\leq l$ such that $f\in Y_t$ and $M_t\neq \phi$. \maxp $f=y_i$.}\\

\textit{\underline{Case $ 3.1 $:} Assume $(m_i,y_i)\in E(D')$}. Clearly, $f$ gains no new first out-neighbor. We prove that $f$ gains no new second out-neighbor. Assume $y_i\ra a\ra b\ra y_i$ in $T'$. Then $(y_i, a)\in E(D)$ and $(a,b)\in E(T)$.\\
	
	\textit{Subcase $  a $}: If $(a,b)\in E(D)$, then clearly \ndn.\\
	
	\textit{Subcase $ b $:} If $(a,b)\in E(D')-E(D)$, then either $(a,b)$ is a convenient orientation w.r.t. $D$ and hence \nd or there is \lose. Thus $\exists$ $ j$ such that $rs, ab\in E[Y_j,M_j]$. Assume $r=y_i$. Then $s=m_i$, $a=m_{i+1}$ and $b=y_{i+1}$. Since $(m_i,y_i)\in E(D')$, then $(y_{i+1}, m_{i+1})\in E(D')$, that is, $(b,a)\in E(D')$, a contradiction. Then $r\neq y_i$. Assume $s=y_i$. Then $a=y_{i+1}$. Hence $y_ia=y_iy_{i+1}$ is a missing edge, contradiction. So $s\neq y_i$. Now we prove that $y_is$ is not a missing edge. If $a \in Y_j$, then $y_ia$ is a missing edge, a contradiction. So  $a \in M_j$ and hence $s \in M_j$. Since $y_ia$ is not a missing edge, then by definition of $G$, $y_is$ is also not a missing edge. Since $y_i\ra a$ in $D$ and $a\notin N^{++}(s)$, then by Lemma \ref{snlemma} we get $b \in N^{++}_D(f)$.\\
	
	\textit{Subcase $ c $}: If $(a,b)\in E(T)-E(D')$, then $(a,b)$ is a convenient orientation w.r.t. $D'$. Whence, \nndn. Then there is a vertex $a'$ such that $y_i\ra a'\ra b$ in $D'$. Thus $(y_i,a')\in D$. This case is already treated in Subcase $i.3.1.a$ and Subcase  $i.3.1.b$.\\
	
\textit{\underline{Case $ 3.2 $:} Assume $(y_i,m_i)\in E(D')$.}\\

\textit{Case $ 3.2.1 $: Assume $i=k$, that is, $f=y_k$}. Clearly, $f$ gains no new out-neighbor. We prove that $f$ gains no new second out-neighbor. Suppose that $f\ra a\ra b\ra f$ in $T'$. Then $(f,a)\in E(D)$ and $(a,b)\in E(T)$.\\

\textit{Subcase $ a $:} If $(a,b)\in E(D)$, then clearly \ndn.\\

\textit{Subcase $  b $:} If $(a,b)\in E(D')-E(D)$, then either $(a,b)$ is a convenient orientation w.r.t. $D$ and hence \nd or there is \lose. Then $\exists$ $ j$ such that $rs, ab\in E[Y_j,M_j]$. Since $f=y_k$, then $r\neq y_k$ and $s\neq y_k$. Since $(y_k,a)\in E(D)$, then $a \notin Y$. Whence,  $a\in M$ and $s\in M$. Since $y_ka$ is not a missing edge, then by definition of $G$, $y_ks$ is also not a missing edge. Since $f\ra a$ in $D$ and $a \notin N^{++}_D(s)$, then by Lemma \ref{snlemma} we get $b \in N^{++}_D(f)$.\\

\textit{Subcase $ c $:} If $(a,b)\in E(T)-E(D')$, then it is a convenient orientation w.r.t. $D'$. But $f\ra a$ in $D$ and thus in $D'$, then \nndn. So there is a vertex $a'$ such that $f\ra a'\ra b$ in $D'$.\\

\textit{Subcase $ c.1 $}: Suppose $a'=m_k$. Since $y_ka$ is not a missing edge of $D$, then $a$ is a whole vertex of $D$ or $a\in M-\{m_k\}$. Since $(a,b)\in E(T)-E(D)$, then $a$ is not whole. Thus $\exists$ $ j$ such that $a \in M_j-\{m_k\}$. The definition of $G$ together with  the facts that  $f=y_k\in Y_t$, $a\in M_j-\{m_k\}$ and $y_ka$ is not missing edge imply that $j\geq t$. Since $(a,b)\notin E(D')$ and $a\in M_j$, then $\exists$ $ \alpha >j$ such that $b \in Y_{\alpha}$. Thus $ba'\in E[Y_{\alpha}, ,M_t]$ with $\alpha > t$. So, by using the definition of $G$, $a'b$ is a missing edge of $D$ and $D'$, a contradiction since  $(a',b)\in E(D')$.\\

\textit{Subcase $ c.2 $}: Suppose $a' \neq m_k$. Whence $(y_k,a')\in E(D)$. But $(a',b)\in E(D')$, then this is already discussed in Subcase $i.3.2.1.a$ and Subcase $i.3.2.1.b$.\\

\textit{Case $ 3.2.2 $: Assume $i<k$}. Clearly, $f=y_i$ gains only $m_i$ as an out-neighbor. We prove that $f$ gains only $y_{i+1}$ as a second out-neighbor. \\
	
	\textit{Subcase $ a $:} Suppose that $f \ra m_i \ra b \ra f$ in $T'$ with $b \neq y_{i+1}$. Then $(y_i, m_i)\in E(D')$ and $(m_i,b)\in E(T)$.\\
	
	\textit{Subcase $ a.1 $:} If $(m_i,b)\in E(D)$. Since $b \neq y_{i+1}$ and $m_ib$ is not a missing edge of $D$, then by definition of $G$, also $m_{i+1}b$ is not a missing edge. Since $m_i\ra b$ in $D$ and $m_{i+1}\notin N^{++}(m_i)$, then we must have $m_{i+1}\ra b$ in $D$. Thus $y_i\ra m_{i+1}\ra b$ in $D$.\\
	
	\textit{Subcase $ a.2 $}: If $(m_i, b)\in E(D')-E(D)$. Then $m_ib=m_iy_i$ and hence $b=y_i$, a contradiction. So this case does not hold.\\
	
	\textit{Subcase $ a.3 $:} If $(m_i,b)\in E(T)-E(D')$. Then $(m_i,b)$ is a convenient orientation w.r.t. to $D'$. Since $y_i\ra m_i$ in $D'$, then $b \in N^{++}_{D'}(y_i)$. Then there is a vertex $a'$ such that $y_i\ra a'\ra b$ in $D'$. If $a'=m_i$, then $(m_i,b) \in E(D')$, a contradiction. Thus $a'\neq m_i$ and so  $(y_i,a')\in E(D)$ and $(a',b)\in E(D')$.\\
	
	\textit{Subcase $ a.3.1 $:} If $(a',b)\in E(D)$, then $y_i\ra a'\ra b $ in $D$.\\
	
	\textit{Subcase $ a.3.2 $:} If $(a',b)\in E(D')-E(D)$, then either $(a',b)$ is a convenient orientation w.r.t. $D$ and hence \nd or there is $rs \rightarrow a'b $ in $\Delta$, namely $s \rightarrow b$ and $ a' \notin N^{++}_D(s)$. If $rs=m_iy_i$, then $a'b=m_{i+1}y_{i+1}$. But $b\neq y_{i+1}$, then $b=m_{i+1}$. Since $(y_i,m_i)\in E(D')$, then $(m_{i+1},y_{i+1})\in E(D')$, that is,  $(b,a')\in E(D')$,  a contradiction. Now we claim that $y_is$ is not a missing edge of $D$. Since $y_ia'$ is not  a missing edge,  then $a'\notin Y$.  Whence, $a'\in M$ and $s\in M$. Therefore, using the definition of $G$ and the fact that $y_ia'$ is not  a missing edge, we reach our claim. Since $y_i\ra a'$ in  $D$ and $ a'\notin N^{++}_D(s)$, then  by Lemma \ref{snlemma} we get $b \in N^{++}_D(f)$. \\
	
	\textit{Subcase $ b $:} Assume $y_i\ra a\ra b \ra y_i$  in $T'$ with $a \neq m_i$ and $b \neq y_{i+1}$. Then $(y_i,a)\in E(D)$ and $(a,b)\in E(T)$. If $(a,b)\in E(D')$, then this is already treated in Subcases $ a.3.1 $ and $ a.3.2 $. Else if $(a,b)\in E(T)-E(D')$, then it is a convenient orientation w.r.t. $D'$ and hence $b \in N^{++}_{D'}(y_i)$. Then there is a vertex $a'$ such that $y_i\ra a'\ra b$ in $D'$. If $a'= m_i$, this is already treated in Subcase $i.3.2.2.a.1 $ and Subcase $i.3.2.2. a.2 $. Else if $a' \neq m_i$, this is already treated in  Subcase $i.3.2.2.a.3.1 $ and Subcase $i.3.2.2.a.3.2 $.\\
	
\textbf{\underline{Case $ 4 $:}} \textit{$\exists$ $ 1\leq t\leq l+1$ such that $f\in Y_t$ such that $M_t=\phi$}. Clearly, $f$ gains no new out-neighbor. We prove it gains no new second out-neighbor. Suppose $f\ra a\ra b\ra f$ in $T'$. Then $(f,a)\in E(D)$ and $(a,b)\in E(T)$. We consider the following cases.\\
	
	\textit{Subcase $ a $:} If $(a,b) \in E(D)$, then clearly \ndn.\\
	
	\textit{Subcase $ b $}: If $(a,b)\in E(D')-E(D)$, then either $(a,b)$ is a convenient orientation w.r.t. $D$ and hence \nd or there is \lose. Then $\exists$ $ j$ such that $rs, ab\in E[Y_j,M_j]$. Since $fa$ is not a missing edge of $D$, then $a\notin Y$. Hence $a\in M_j$ and $s\in M_j$. Since $a,s\in M_j$ and $fa$ is not a missing edge, then also $fs$ is not a missing edge. Since $f\ra a$ in $D$ and $a \notin N^{++}_D(s)$, then by Lemma \ref{snlemma} we get $b \in N^{++}_D(f)$.\\
	
	\textit{Subcase $ c $}: If $(a,b)\in E(T)-E(D')$, then $(a,b)$ is a convenient orientation w.r.t. $D'$. But $f \ra a$ in $D$ and $D'$, then $b \in N^{++}_{D'}(f)$. So there is a vertex $a'$ such that $f \ra a' \ra b$ in $D'$. Then $(f,a') \in E(D)$. But this is already treated in Subcase $i.4.a$ and Subcase $i.4.b$.\\	
	
\textbf{\underline{Case $ 5 $:}} \textit{$f\in Y_{l+2}$}. Exactly same as Case $i. 4 $, with only one difference in Subcase $ b $. The difference is that   $fs$ is not a missing edge in Subcase $ i.5.b $, because $E[Y_{l+2}, M_j]=\phi$ by definition of $G$, while in Subcase $i. 4.b $ we had to prove it.\\

\textbf{\underline{Case $ 6 $:}} \textit{$f\in V(G)-(Y\cup M)$}. Clearly, $f$ gains no new out-neighbor. We will prove that it gains no new second out-neighbor. Suppose $f\ra a\ra b\ra f$ in $T'$. Then $(f,a)\in E(D)$ and $(a,b)\in E(T)$. We consider the following subcases.\\

\textit{Subcase $  a $:} If $(a,b)\in E(D)$, then clearly \ndn.\\

\textit{Subcase $ b $:} If $(a,b)\in E(D')-E(D)$, then either $(a,b)$ is a convenient orientation w.r.t. $D$ and hence \nd or there is \lose. Thus $\exists$ $ j$ such that $rs, ab \in E[Y_j,M_j]$. If $a \in Y_j$, then $fa$ is a missing edge of $D$, a contradiction. So $a \in M_j$ and hence $s\in M_j$. Then $fs$ is not a missing edge. Since $f\ra a $ in $D$ and $a\notin N^{++}_D(s)$, then by Lemma \ref{snlemma} we get $b \in N^{++}_D(f)$.\\

\textit{Subcase $ c $:} If $(a,b)\in E(T)-E(D')$. Then $(a,b)$ is a convenient orientation w.r.t. $D'$. But $f\ra a$ in $D$ and so $D'$, then $b\in N^{++}_{D'}(f)$. Then there is a vertex $a'$ such that $f\ra a'\ra b$ in $D'$. Then $(f,a')\in E(D)$ and $(a',b)\in E(D')$. But this is already discussed in Subcase $i.6. a $ and Subcase $i.6. b $.\\

Therefore,  $f$ has the SNP in $D$ in all cases. This completes the proof of Case $ i $.\\

\noindent \textbf{\underline{Case $ ii $:}} \textit{$\Delta[E(C)]$ contains exactly one arc, say $uv\ra xy$}.\\

\noindent	Assume without loss of generality that $(u,v)$ is a convenient orientation of the good missing edge $uv$. We add to $D$ the arcs $(u,v)$ and $(x,y)$, we assign to the good missing edges $xv$, $yz$ and $zu$ a convenient orientation and then we add them to $D$. The obtained oriented graph $D'$ is missing $G'=G-(\cup E[Y_j, M_j] \cup E(C))$ which is a threshold graph. We assign to the missing edges of $D'$ convenient orientations and we add them to $D'$ to get a tournament $T$. Let $L$ be a local median order of $T$ and let $f$ denote its feed vertex. Reorient all the missing edges incident to $f$ towards $f$, except those whose out-degree in $\Delta$ is not zero.  \same . We will prove that $f$ has the SNP in $D$ also. For this purpose, we consider the following cases.\\

\textbf{\underline{Case $ 1 $:}} \textit{$f$ is a whole vertex.} This is the same as Case $i.1$.\\

\textbf{\underline{Case $ 2 $:}} \textit{$\exists$ $ 1\leq t\leq l$ such that $f\in M_t$.} Exactly same as Case $i.2$, with only one difference in the subcases when $f \ra a \ra b$ with $(f,a) \in E(D)$ and $(a,b) \in E(D')-E(D)$ and it is not convenient w.r.t. $D$. Such subcase we call it unsteady. As usual, since $(a,b)$ is not convenient w.r.t. $D$, then there is  \lose. The difference is that in the unsteady subcases of Case $ ii.2 $ either  $rs, ab \in E[Y_j, M_j]$ for some $j$ or $rs=uv, ab=xy$. If $rs, ab \in E[Y_j, M_j]$ for some $j$, we proceed exactly in the same way as in the unsteady subcases of Case $i.2$.  Else if $(r,s)=(u,v)$ and $( a,b)=(x,y)$, then $fs$ is not a missing edge because $E[M, C] = \phi$. Since $f\ra a$ in $D$ and $a \notin N^{++}_D(s)$, then by Lemma \ref{snlemma} we get $b \in N^{++}_D(f)$.\\

\textbf{\underline{Case $ 3 $:}} \textit{$\exists$ $ 1\leq t\leq l$ such that $f\in Y_t$ and $M_t\neq \phi$}. Exactly same as Case $i.3$, with only one difference in the subcases when $f \ra a \ra b$ with $(f,a) \in E(D)$ and $(a,b) \in E(D')-E(D)$ and it is not convenient w.r.t. $D$. As usual, since $(a,b)$ is not convenient w.r.t. $D$, then there is  \lose. The difference is that in the unsteady subcases of Case $ ii.3 $ there are two cases to be consider: Either  $rs, ab \in E[Y_j, M_j]$ for some $j$, or $rs=uv$ and  $ab=xy$. If $rs, ab \in E[Y_j, M_j]$ for some $j$, we proceed exactly in the same way as in the unsteady subcases of Case $i.3$.  Else if $(r,s)=(u,v)$ and $( a,b)=(x,y)$, then $fa=fx$ is  a missing edge because $G[Y \cup C]$ is a complete split graph, a contradiction. \\

\textbf{\underline{Case $ 4 $:}} \textit{$\exists$ $ 1\leq t\leq l+1$ such that $f\in Y_t$ such that $M_t=\phi$}. Exactly same as Case $ i.4 $, with only one difference in Subcase $ b $. The difference is that in Subcase $ii.4.b$ there are two possibilities for the edges $rs, ab$: Either  $rs, ab \in E[Y_j, M_j]$ for some $j$,  or $(r,s)=(u,v)$ and  $(a,b)=(x,y)$. The first case is treated in Subcase $i.4.b$. However, the second case does not exist since otherwise  $fa=fx$ is  a missing edge because $G[Y \cup C]$ is a complete split graph, which contradicts the fact that $(f,a) \in E(D)$. \\

\textbf{\underline{Case $ 5 $:}} \textit{$f\in Y_{l+2}$}. Exactly same as Case $ i.4 $, with two differences in Subcase $ b $. The first difference is that in Subcase $ii.5.b$ there are two possibilities for the edges $rs, ab$: Either  $rs, ab \in E[Y_j, M_j]$ for some $j$,  or $(r,s)=(u,v)$ and  $(a,b)=(x,y)$. The first case is already treated in Subcase $i.4.b$. However, the second case does not exist since otherwise  $fa=fx$ is  a missing edge because $G[Y \cup C]$ is a complete split graph, which contradicts the fact that $(f,a) \in E(D)$. The second  difference is that   $fs$ is not a missing edge in Subcase $ ii.5.b $, because $E[Y_{l+2}, M_j]=\phi$ by definition of $G$, while in Subcase $ i.4.b $ we had to prove it.\\

\textbf{\underline{Case $ 6 $:}} \textit{$f=u$}. Clearly, $u$ gains only $v$ as a new first out-neighbor. We prove that it gains only $y$ as a new second out-neighbor.\\

\underline{\textit{Case  $6.1$}:} \textit{Suppose that $u\ra v\ra b\ra u$ in $T'$ with $b\neq y$}. Then $(v,b) \in E(T)$. Note that $b\neq x$ because $u\ra x$ in $D$. \\

\textit{Subcase $a$}: If $(v,b)\in E(D)$, then either  $b=z$, $b \in S$ or $b$ is a whole vertex. Then by the losing relation $uv\ra xy$, we get $b\in N^{++}_D(u)$.\\

\textit{Subcase $b$}: If $(v,b)\in E(D')-E(D)$,  then $b = x$ or $b = u$, a contradiction. Thus this case does not exist.\\

\textit{Subcase $c$}: If $(v,b)\in E(T)-E(D')$, then $b\in K$ and $(v,b)$ is a convenient orientation w.r.t. $D'$. Then there exists $v'$ such that $u\ra v'\ra b\ra u$ in $D'$. Since $(v',b)\in E(D')$ and $b\in K$, then $v'\notin C$. Since $v'\notin C$ and $(f,v')\in E(D')$, then $(f,v')\in E(D)$.\\

\textit{Subcase $c.1$}: If  $(v',b)\in E(D)$, then $b\in N^{++}_D(f)$.\\

\textit{Subcase $c.2$}:  If  $(v',b)\in E(D')-E(D)$, then  $(v',b)$ is a convenient orientation w.r.t $D$ and hence $b\in N^{++}_D(f)$ or there is $rs \rightarrow v'b $ in $\Delta$, namely $s \rightarrow b$ and $ v' \notin N^{++}_D(s)$. Since $v' \notin C$, then $\exists$ $j$ such that  $rs, v'b \in E[Y_j, M_j]$. Since $b \in K$, then $b \in Y_j$. Whence, $v' \in M_j$ and $s \in M_j$. Thus $fs$ is not a missing edge of $D$ by definition of $G$. But $(f,v')\in E(D)$ and $v' \notin N^{++}_D(s)$, then  by Lemma \ref{snlemma} we get $b\in N^{++}_D(f)$.\\

\underline{\textit{Case $6.2$}:}\textit{ Suppose that $u\ra a\ra b\ra u$ in $T'$ with $a\neq v$ and $b\neq y$}. Then $(u,a)\in  E(D)$  and $(a,b)\in  E(T)$. Note that $a\notin K\cup \{u,v,y,z\}$ and $b\notin \{u,v,x, y\}$.\\

\textit{Subcase $a$}: If $(a,b)\in E(D)$, $b\in N^{++}_D(u)$.\\

\textit{Subcase $b$}: If $(a,b)\in E(D')-E(D)$, then either $ab \in E(C)$ or $ ab \in E[Y_j,M_j]$ for some $j$. If $ab \in E(C)$, then $(a,b)=(x,z)$ because $a\notin \{u,v,y,z\}$ and $b\notin \{u,v,y,x\}$. Thus $xz$ is a missing edge of $D$, a contradiction. It follows that $ ab \in E[Y_j,M_j]$ for some $j$. Then either $(a,b)$ is a convenient orientation w.r.t. $D$ and hence $b\in N^{++}_D(u)$ or there is $rs \rightarrow ab $ in $\Delta$, namely $s \rightarrow b$ and $ a \notin N^{++}_D(s)$. Since $ a\notin K$, then $ a \in M_j$ and thus $ s \in M_j$. So by definition of $G$, $us$ is not a missing edge of $D$. But $(u,a)\in  E(D)$ and $ a \notin N^{++}_D(s)$, then  by Lemma \ref{snlemma} we get $b\in N^{++}_D(u)$.\\

\textit{Subcase $c$}: If $(a,b)\in E(T)-E(D')$, then $b\in K$ and $(a,b)$ is a convenient orientation w.r.t. $D'$. Since $(u,a)\in  E(D)$ and so in $D'$, there exists $v'$ such that $u\ra v'\ra b$ in $D'$. Since $(v',b)\in E(D')$ and $b\in K$, then $v'\notin C$. Since $v'\notin C$ and $(u,v')\in E(D')$, then $(u,v')\in E(D)$. But this is already treated in Subcase $ii.6.1.c.1$ and Subcase $ii.6.1.c.2$.\\

\textbf{\underline{Case $ 7$:}} \textit{$f \in C-\{u\}$}. It is clear that $f$ gains no new first out-neighbor. We will prove that it gains no new second out-neighbor. Suppose that $f\ra a\ra b\ra f$ in $T'$. Then $(f,a)\in  E(D)$, $(a,b)\in E(T)$  and $ a \notin K$.\\

\textit{Subcase $a$}: If $(a,b)\in E(D)$, $b\in N^{++}_D(u)$.\\

\textit{Subcase $b$}: If $(a,b)\in E(D')-E(D)$, then either $(a,b)$ is a convenient orientation w.r.t. $D$ and hence  $b \in N^{++}_D(f)$ or there is there is $rs \rightarrow ab $ in $\Delta$, namely $s \rightarrow b$ and $ a \notin N^{++}_D(s)$. If $(r,s)= (u,v)$ and $(a,b)=(x,y)$, then $a=x$ and $b=y$ and so $f \notin \{x,y\}$. Note that $f \neq v$, since otherwise $fa=vx$ is a missing edge of $D$, a contradiction. Since  $f \in C-\{u,v,x,y\}$, then $f=z$ and hence $fs=zv$ is not a missing edge of $D$. Else if $rs, ab \in E[Y_j, M_j]$ for some $j$, then $r \neq f$ and $s \neq f$. Since $ a \notin K$, then $a \notin Y_j$ and so $a \in M_j$. Whence, $s \in M_j$. Thus $fs$ is not missing edge of $D$ by definition of $G$. Therefore, by the losing relation $rs \ra ab$ in $\Delta$, we get $b \in N^{++}_D(f)$.\\

\textit{Subcase $c$}: If $(a,b)\in E(T)-E(D')$, then $b \in K$ and $(a,b)$ is a convenient orientation w.r.t. $D'$. So  there is $v'$ such that $f\ra v' \ra b$ in $D'$. Since $(v',b)\in E(D')$ and $b\in K$, then $v'\notin C$. Since $v'\notin C$ and $(f,v')\in E(D')$, then $(f,v')\in E(D)$. But this is already treated in Subcase $ii.6.1.c.1$ and Subcase $ii.6.1.c.2$.\\

\textbf{\underline{Case $ 8$:}} \textit{$f\in V(G)-(Y\cup M \cup C) = A \cup (X-X_1)= A \cup (X-Y_1) $}.  Exactly same as Case $ i.6 $, with only one difference in Subcase $ b $. The difference is that in Subcase $ii.8.b$ there are two possibilities for the edges $rs, ab$: Either  $rs, ab \in E[Y_j, M_j]$ for some $j$,  or $(r,s)=(u,v)$ and  $(a,b)=(x,y)$. The first case is treated in Subcase $i.6.b$. However, for the  case $(r,s)=(u,v)$ and  $(a,b)=(x,y)$, $f$ must belong to $A$ since otherwise $fa=fx$ is a missing edge because $G[X \cup C]$  is a complete split graph, which contradicts the fact that $(f,a) \in E(D)$. Thus $fs=fv$ is not missing edge of $D$ because $E[A, C]=\phi$ by definition of $G$. Since $f\ra a$ in $D$ and $a \notin N^{++}_D(s)$, then by Lemma \ref{snlemma} we get $b \in N^{++}_D(f)$.\\

 Therefore,  $f$ has the SNP in $D$   when $\Delta$ has  exactly one arc between the edges of $C$.\\

\noindent \textbf{\underline{Case $ iii $:}} \textit{Suppose that $\Delta[E(C)]$ has exactly two arcs , say $uv\ra xy$ and $vx\ra yz$.}\\

\noindent Then $u\ra x\ra z\ra v\ra y$ in $D$ and $uv, vx, uz$ are good missing edges. Assume without loss of generality that $(u,v)$ is a convenient orientation w.r.t. $D$. Add the arcs $(u,v)$ and $(x,y)$ to $D$. If $(v,x)$ is a convenient orientation of $vx$, then add the arcs $(v,x)$ and $(y,z)$. Otherwise, add the arcs $(x,v)$ and $(z,y)$. Assign to $uz$ a convenient orientation and add it to $D$. The obtained oriented graph $D'$ is missing  $G'=G-(\cup E[Y_j, M_j] \cup E(C))$ which is a threshold graph.  So all the missing edges of $D'$ are good. We assign to them a convenient orientation and we add them to get a tournament $T$. Let $L$ be a local median order of $T$ and let $f$ denote its feed vertex. Reorient all the missing edges incident to $f$ towards $f$ except  those whose out-degree in $\Delta$ is not zero. The same order $L$ is again a local median order of the obtained tournament $T'$ and $f$ has the SNP in $T'$. We will prove that $f$  has the SNP in $D$ also. We have the following cases.\\

\textbf{\underline{Case $ 1 $:}} \textit{$f$ is a whole vertex.} This is the same as Case $i.1$.\\

\textbf{\underline{Case $ 2 $:}} \textit{$\exists$ $ 1\leq t\leq l$ such that $f\in M_t$.} Exactly same as Case $i.2$, with only one difference in the unsteady subcases, that is, in the subcases where $f \ra a \ra b$ with $(f,a) \in E(D)$ and $(a,b) \in E(D')-E(D)$ and it is not convenient w.r.t. $D$.  As usual, since $(a,b)$ is not convenient w.r.t. $D$, then there is  \lose. The difference is that in the unsteady subcases of Case $ iii.2 $  either  $rs, ab \in E[Y_j, M_j]$ for some $j$ or $rs, ab \in E(C)$. If $rs, ab \in E[Y_j, M_j]$ for some $j$, we proceed exactly in the same way as in the unsteady subcases of Case $i.2$.  Else if $rs, ab \in E(C)$ (the possible cases are: $(r,s)= (u,v)$ and $ (a,b)=(x,y)$, $(r,s)= (v,x)$ and $ (a,b)=(y,z)$ if $ (v,x) $ is a convenient orientation of $vx$ or  $(r,s)= (x,v)$ and $(a,b)=(z,y)$ if $ (x,v) $ is a convenient orientation of $vx$), then $fs$ is not a missing edge because $E[M, C] = \phi$.  Since $f\ra a$ in $D$ and $a \notin N^{++}_D(s)$, then by Lemma \ref{snlemma} we get $b \in N^{++}_D(f)$.\\

\textbf{\underline{Case $ 3 $:}} \textit{$\exists$ $ 1\leq t\leq l$ such that $f\in Y_t$ and $M_t\neq \phi$}. Exactly same as Case $i.3$, with only one difference in the subcases when $f \ra a \ra b$ with $(f,a) \in E(D)$ and $(a,b) \in E(D')-E(D)$ and it is not convenient w.r.t. $D$. As usual, since $(a,b)$ is not convenient w.r.t. $D$, then there is  \lose. The difference is that in the unsteady subcases of Case $ iii.3 $ there are two cases to be consider: Either  $rs, ab \in E[Y_j, M_j]$ for some $j$, or $rs, ab \in E(C)$. If $rs, ab \in E[Y_j, M_j]$ for some $j$, we proceed exactly in the same way as in the unsteady subcases of Case $i.3$.  Else if $rs, ab \in E(C)$, then $fa$ is  a missing edge because $G[Y \cup C]$ is a complete split graph, a contradiction. Thus this case does not exist.\\

\textbf{\underline{Case $ 4 $:}} \textit{$\exists$ $ 1\leq t\leq l+1$ such that $f\in Y_t$ such that $M_t=\phi$}. Exactly same as Case $ i.4 $, with only one difference in Subcase $ b $. The difference is that in Subcase $iii.4.b$ there are two possibilities for the edges $rs, ab$: Either  $rs, ab \in E[Y_j, M_j]$ for some $j$,  or $rs, ab \in E(C)$. The first case is treated in Subcase $i.4.b$. However, the second case does not exist since otherwise  $fa$ is  a missing edge because $G[Y \cup C]$ is a complete split graph, which contradicts the fact that $(f,a) \in E(D)$. This means that  this case does not exist. \\

\textbf{\underline{Case $ 5 $:}} \textit{$f\in Y_{l+2}$}. Exactly same as Case $ i.4 $, with two differences in Subcase $ b $. The first difference is that in Subcase $iii.5.b$ there are two possibilities for the edges $rs, ab$: Either  $rs, ab \in E[Y_j, M_j]$ for some $j$,  or $rs, ab \in E(C)$. The first case is already treated in Subcase $i.4.b$. However, the second case does not exist since otherwise  $fa$ is  a missing edge because $G[Y \cup C]$ is a complete split graph, which contradicts the fact that $(f,a) \in E(D)$. The second  difference is that   $fs$ is not a missing edge in Subcase $ iii.5.b $, because $E[Y_{l+2}, M_j]=\phi$ by definition of $G$, while in Subcase $ i.4.b $ we had to prove it.\\

\textbf{\underline{Case $ 6$:}} \textit{$f=u$}. Clearly, $f$ gains only $v$ as a new first out-neighbor and  $y$  as a new second out-neighbor. We prove it gains only $y$ as a second out-neighbor.\\

\underline{\textit{Case $6.1$}:} \textit{Suppose that $u\ra v\ra b\ra u$ in $T'$ with $b\neq y$}. Then $(v,b) \in E(T)$.  Since $x$ and $z$ are first and second out-neighbors of $u$ in $D$ respectively, then we may assume that $b\notin C$ and hence $(v,b)\notin E(D')-E(D)$.\\

\textit{Subcase $a$}: If $(v,b)\in E(D)$, then either  $b \in S$ or $b$ is a whole vertex. Then by the losing relation $uv\ra xy$, we get $b\in N^{++}_D(u)$.\\

\textit{Subcase $b$}: If $(v,b)\in E(T)-E(D')$, then $b\in K$ and $(v,b)$ is a convenient orientation w.r.t. $D'$. But this is exactly the same as Subcase $ii.6.1.c$.\\

\underline{\textit{Case $6.2:$}} \textit{Suppose that $u\ra a\ra b\ra u$ in $T'$ with $a\neq v$ and $b\neq y$}. Thus $(a,b)\in  E(T)$.  Since $a\neq v$, then $(u,a)\in E(D)$ and hence $a\notin K$. Since $x$ and $z$ are first and second out-neighbors of $u$ in $D$ and $u\ra v$ in $T'$, then we may assume $b\notin C$.\\

\textit{Subcase $a$}: If $(a,b)\in E(D)$, $b\in N^{++}_D(u)$.\\

\textit{Subcase $b$}: If $(a,b)\in E(D')-E(D)$, then either $ab \in E(C)$ or $ ab \in E[Y_j,M_j]$ for some $j$. If $ab \in E(C)$, then $b \in C$, a contradiction. Thus $ ab \in E[Y_j,M_j]$ for some $j$. It follows that either $(a,b)$ is a convenient orientation w.r.t. $D$ and hence $b\in N^{++}_D(u)$ or there is $rs \rightarrow ab $ in $\Delta$, namely $s \rightarrow b$ and $ a \notin N^{++}_D(s)$. Since $ a\notin K$, then $ a \in M_j$ and thus $ s \in M_j$. So by definition of $G$, $us$ is not a missing edge of $D$. But $(u,a)\in  E(D)$ and $ a \notin N^{++}_D(s)$, then  by Lemma \ref{snlemma} we get $b\in N^{++}_D(u)$.\\

\textit{Subcase $c$}: If $(a,b)\in E(T)-E(D')$, then $b\in K$ and $(a,b)$ is a convenient orientation w.r.t. $D'$. Since $(u,a)\in  E(D)$ and so in $D'$, there exists $v'$ such that $u\ra v'\ra b$ in $D'$. Since $(v',b)\in E(D')$ and $b\in K$, then $v'\notin C$. Since $v'\notin C$ and $(u,v')\in E(D')$, then $(u,v')\in E(D)$. But this is already treated in Subcase $iii.6.2.a$ and Subcase $iii.6.2.b$.\\

\textbf{\underline{Case $ 7$:}} \textit{$f=v$}. Here there are two cases to be consider.\\

\textit{\underline{Case $ 7.1$:} Assume $(v,x)\in E(D')$.} Then $v$ gains $x$ as first out-neighbor and  $z$ as a second out-neighbor. This case is similar to Case $iii.6$.\\

\textit{\underline{Case $ 7.2$:} Assume $(x,v)\in E(D')$.} Then  $(z,y)\in E(D')$. Clearly, $v$ gains no new first out-neighbor. We prove that it gains no new second out-neighbor. Suppose $f\ra a\ra b\ra f$ in $T'$. Then $(f,a)\in E(D)$ and hence $a\notin K$.\\

\textit{Subcase $a$}: If $(a,b)\in E(D)$, $b\in N^{++}_D(f)$.\\

\textit{Subcase $b$}: If $(a,b)\in E(D')-E(D)$, then  either $(a,b)$ is a convenient orientation w.r.t. $D$ and hence $b\in N^{++}_D(v)$ or there is $rs \rightarrow ab $ in $\Delta$, namely $s \rightarrow b$ and $ a \notin N^{++}_D(s)$. So either $rs, ab \in E(C)$ or $ rs, ab \in E[Y_j,M_j]$ for some $j$. If $ab \in E(C)$, then $(a,b)=(x,y)$ or $(a,b)=(z,y)$ and hence $b=y$, which is impossible because $b\ra f$ in $T'$ while $f\ra y$ in $D$. Thus $rs, ab \in E[Y_j,M_j]$ for some $j$.  Since $ a\notin K$, then $ a \in M_j$ and thus $ s \in M_j$. So by definition of $G$, $fs$ is not a missing edge of $D$. But $(f,a)\in  E(D)$ and $ a \notin N^{++}_D(s)$, then  by Lemma \ref{snlemma} we get $b\in N^{++}_D(f)$.\\

\textit{Subcase $c$}: If $(a,b)\in E(T)-E(D')$, then $b\in K$ and $(a,b)$ is a convenient orientation w.r.t. $D'$. Since $(f,a)\in  E(D)$ and so in $D'$, there exists $v'$ such that $f\ra v'\ra b$ in $D'$. Since $(v',b)\in E(D')$ and $b\in K$, then $v'\notin C$. Since $v'\notin C$ and $(f,v')\in E(D')$, then $(f,v')\in E(D)$. But this is already treated in Subcase $iii.7.2.a$ and Subcase $iii.7.2.b$.\\

\textbf{\underline{Case $ 8$:}} \textit{$f=x$}. Here there are two cases to be consider.\\

\textit{\underline{Case $ 8.1$:} Assume $(v,x)\in E(D')$.} Clearly, $x$ gains no new first out-neighbor. We prove it gains no new second out-neighbor. Suppose $x\ra a\ra b\ra x$ in $T'$. Then $(x,a)\in E(D)$, $(a,b) \in E(T)$ and $a\notin K$. Note that $a\notin \{x,y,u,v\}$ and we may assume $b\notin \{x,z,v\}$.\\

\textit{Subcase $a$}: If $(a,b)\in E(D)$, $b\in N^{++}_D(x)$.\\

\textit{Subcase $b$}: If $(a,b)\in E(D')-E(D)$, then  either $(a,b)$ is a convenient orientation w.r.t. $D$ and hence $b\in N^{++}_D(x)$ or there is $rs \rightarrow ab $ in $\Delta$, namely $s \rightarrow b$ and $ a \notin N^{++}_D(s)$. So either $rs, ab \in E(C)$ or $ rs, ab \in E[Y_j,M_j]$ for some $j$. If $ab \in E(C)$, then $(a,b)=(x,y)$ or $(a,b)=(y,z)$ and hence $a \in \{x,y\}$, a contradiction. Thus $rs, ab \in E[Y_j,M_j]$ for some $j$.  Since $ a\notin K$, then $ a \in M_j$ and thus $ s \in M_j$. So by definition of $G$, $fs$ is not a missing edge of $D$. But $(x,a)\in  E(D)$ and $ a \notin N^{++}_D(s)$, then  by Lemma \ref{snlemma} we get $b\in N^{++}_D(f)$.\\

\textit{Subcase $c$}: If $(a,b)\in E(T)-E(D')$, then $b\in K$ and $(a,b)$ is a convenient orientation w.r.t. $D'$. Since $(x,a)\in  E(D)$ and so in $D'$, there exists $v'$ such that $x\ra v'\ra b$ in $D'$. Since $(v',b)\in E(D')$ and $b\in K$, then $v'\notin C$. Since $v'\notin C$ and $(x,v')\in E(D')$, then $(x,v')\in E(D)$. But this is already treated in Subcase $iii.8.1.a$ and Subcase $iii.8.1.b$.\\

\textit{\underline{Case $ 8.2$:} Assume $(x,v)\in E(D')$.} Clearly, $x$ gains only $v$ as a first out-neighbor and $y$ as a second out-neighbor. We prove it gains only $y$ as a second out-neighbor. Note that $z$ and $v$ are first and second out-neighbors of $x$ in $D$,  $(v,u)\notin E(T)$ and $(x,u)\notin E(T)$.\\

\textit{Subcase $a$}: Suppose that $x\ra v\ra b\ra x$ in $T'$ with $b\neq y$. Then $(v,b) \in E(T)$.  By the previous note, we may assume that $b\notin C$ and hence $(v,b)\notin E(D')-E(D)$.\\

\textit{Subcase $a.1$}: If $(v,b)\in E(D)$, then either  $b \in S$ or $b$ is a whole vertex. Then by the losing relation $xv \ra zy$, we get $b\in N^{++}_D(x)$.\\

\textit{Subcase $a.2$}: If $(v,b)\in E(T)-E(D')$, then $b\in K$ and $(v,b)$ is a convenient orientation w.r.t. $D'$. But this exactly the same as Subcase $ii.6.a.3$.\\

\textit{Subcase $b$}: Suppose $x\ra a\ra b\ra x$ in $T'$ with $b\neq y$ and $a\neq v$. Then $(x,a)\in E(D)$ and thus $a\notin K\cup \{x,y,v,u\}$. Note that  suppose $b\notin \{x,y,v,z\}$. We argue exactly as in  Subcase $ iii.7.2$.\\

\textbf{\underline{Case $ 9$:}} \textit{$ f=y$.} Clearly, $f$ gains no new first out-neighbor. We prove that it gains no new second out-neighbor. Note that $(z,y)\in E(T')$ and  $u$ and $x$ are first and second out-neighbors of $y$, respectively. Suppose that $f\ra a\ra b\ra f$ in $T'$. Then $(f,a)\in E(D)$ and $a\notin K \cup \{x,y, v,z\}$.\\

\textit{Subcase $a$}: If $(a,b)\in E(D)$, then  $b\in N^{++}_D(f)$.\\

\textit{Subcase $b$}: If $(a,b)\in E(D')-E(D)$, then  either $(a,b)$ is a convenient orientation w.r.t. $D$ and hence $b\in N^{++}_D(f)$ or there is $rs \rightarrow ab $ in $\Delta$, namely $s \rightarrow b$ and $ a \notin N^{++}_D(s)$. So either $rs, ab \in E(C)$ or $ rs, ab \in E[Y_j,M_j]$ for some $j$. If $ab \in E(C)$, then $(a,b)=(x,y)$ or $(a,b)=(y,z)$ or $(a,b)=(z,y)$ and hence $a \in \{x,y,z\}$, a contradiction. Thus $rs, ab \in E[Y_j,M_j]$ for some $j$.  Since $ a\notin K$, then $ a \in M_j$ and thus $ s \in M_j$. So by definition of $G$, $fs$ is not a missing edge of $D$. But $(f,a)\in  E(D)$ and $ a \notin N^{++}_D(s)$, then  by Lemma \ref{snlemma} we get $b\in N^{++}_D(f)$.\\

\textit{Subcase $c$}: If $(a,b)\in E(T)-E(D')$, then $b\in K$ and $(a,b)$ is a convenient orientation w.r.t. $D'$. Since $(f,a)\in  E(D)$ and so in $D'$, there exists $v'$ such that $f\ra v'\ra b$ in $D'$. Since $(v',b)\in E(D')$ and $b\in K$, then $v'\notin C$. Since $v'\notin C$ and $(f,v')\in E(D')$, then $(f,v')\in E(D)$. But this is already treated in Subcase $iii.9.a$ and Subcase $iii.9.b$.\\

\textbf{\underline{Case $ 10$:}} \textit{$ f=z$.} Exactly same as Case $iii.9$ with difference that $yz$ is reoriented so that $(y,z)\in E(T')$.\\

\textbf{\underline{Case $  11$:}} \textit{$f\in V(G)-(Y\cup M \cup C) = A \cup (X-X_1)= A \cup (X-Y_1) $}.  Exactly same as Case $ i.6 $, with only one difference in Subcase $ b $. The difference is that in Subcase $iii.11.b$ there are two possibilities for the edges $rs, ab$: Either  $rs, ab \in E[Y_j, M_j]$ for some $j$,  or $rs, ab \in E(C)$. The first case is treated in Subcase $i.6.b$. However, for the second case, $f$ must belong to $A$ since otherwise $fa$ is a missing edge because $G[X \cup C]$  is a complete split graph, which contradicts the fact that $(f,a) \in E(D)$. Thus $fs$ is not missing edge of $D$ because $E[A, C]=\phi$ by definition of $G$. Since $f\ra a$ in $D$ and $a \notin N^{++}_D(s)$, then by Lemma \ref{snlemma} we get $b \in N^{++}_D(f)$.\\

Therefore, due to all above discussions, $f$ has the SNP in $D$.\\

\noindent \textbf{\underline{Case $ iv $:}} \textit{Suppose that $\Delta[E(C)]$ has exactly two arcs,   say $uv\ra xy \ra zu$.}\\

\noindent Then  $v\ra y\ra u\ra x\ra z$ in $D$ and $uv, yz, vx$ are good missing edges. Assume without loss of generality that $(u,v)$ is a convenient orientation w.r.t. $D$. Add the arcs $(u,v)$, $(x,y)$ and $(z,u)$to $D$. Assign to the good missing edges $vx$ and $zy$ a convenient orientation and add them to $D$. The obtained oriented graph $D'$ is missing  $G'=G-(\cup E[Y_j, M_j] \cup E(C))$ which is a threshold graph.  So all the missing edges of $D'$ are good. We assign to them a convenient orientation and we add them to get a tournament $T$. Let $L$ be a local median order of $T$ and let $f$ denote its feed vertex. Reorient all the missing edges incident to $f$ towards $f$ except  those whose out-degree in $\Delta$ is not zero. The same order $L$ is again a local median order of the obtained tournament $T'$ and $f$ has the SNP in $T'$. We will prove that $f$  has the SNP in $D$. We have the following cases.\\

\textbf{\underline{Case $ 1 $:}} \textit{$f$ is a whole vertex.} This is the same as Case $i.1$.\\

\textbf{\underline{Case $ 2 $:}} \textit{$\exists$ $ 1\leq t\leq l$ such that $f\in M_t$.} Exactly same as Case $i.2$, with only one difference in the unsteady subcases, that is, in the subcases where $f \ra a \ra b$ with $(f,a) \in E(D)$ and $(a,b) \in E(D')-E(D)$ and it is not convenient w.r.t. $D$.  As usual, since $(a,b)$ is not convenient w.r.t. $D$, then there is  \lose. The difference is that in the unsteady subcases of Case $ iv.2 $  either  $rs, ab \in E[Y_j, M_j]$ for some $j$ or $rs, ab \in E(C)$. If $rs, ab \in E[Y_j, M_j]$ for some $j$, we proceed exactly in the same way as in the unsteady subcases of Case $i.2$.  Else if $rs, ab \in E(C)$ (the possible cases are: $(r,s)= (u,v)$ and $ (a,b)=(x,y)$ or  $(r,s)= (x,y)$ and $ (a,b)=(z,u)$), then $fs$ is not a missing edge because $E[M, C] = \phi$.  Since $f\ra a$ in $D$ and $a \notin N^{++}_D(s)$, then by Lemma \ref{snlemma} we get $b \in N^{++}_D(f)$.\\

\textbf{\underline{Case $ 3 $:}} \textit{$\exists$ $ 1\leq t\leq l$ such that $f\in Y_t$ and $M_t\neq \phi$}. Exactly same as Case $i.3$, with only one difference in the subcases when $f \ra a \ra b$ with $(f,a) \in E(D)$ and $(a,b) \in E(D')-E(D)$ and it is not convenient w.r.t. $D$. As usual, since $(a,b)$ is not convenient w.r.t. $D$, then there is  \lose. The difference is that in the unsteady subcases of Case $ iv.3 $ there are two cases to be consider: Either  $rs, ab \in E[Y_j, M_j]$ for some $j$, or $rs, ab \in E(C)$. If $rs, ab \in E[Y_j, M_j]$ for some $j$, we proceed exactly in the same way as in the unsteady subcases of Case $i.3$.  Else if $rs, ab \in E(C)$, then $fa$ is  a missing edge because $G[Y \cup C]$ is a complete split graph, a contradiction. Thus this case does not exist.\\

\textbf{\underline{Case $ 4 $:}} \textit{$\exists$ $ 1\leq t\leq l+1$ such that $f\in Y_t$ such that $M_t=\phi$}. Exactly same as Case $ i.4 $, with only one difference in Subcase $ b $. The difference is that in Subcase $iv.4.b$ there are two possibilities for the edges $rs, ab$: Either  $rs, ab \in E[Y_j, M_j]$ for some $j$,  or $rs, ab \in E(C)$. The first case is treated in Subcase $i.4.b$. However, the second case does not exist since otherwise  $fa$ is  a missing edge because $G[Y \cup C]$ is a complete split graph, which contradicts the fact that $(f,a) \in E(D)$. This means that  this case does not exist. \\

\textbf{\underline{Case $ 5 $:}} \textit{$f\in Y_{l+2}$}. Exactly same as Case $ i.4 $, with two differences in Subcase $ b $. The first difference is that in Subcase $iv.5.b$ there are two possibilities for the edges $rs, ab$: Either  $rs, ab \in E[Y_j, M_j]$ for some $j$,  or $rs, ab \in E(C)$. The first case is already treated in Subcase $i.4.b$. However, the second case does not exist since otherwise  $fa$ is  a missing edge because $G[Y \cup C]$ is a complete split graph, which contradicts the fact that $(f,a) \in E(D)$. The second  difference is that   $fs$ is not a missing edge in Subcase $ iv.5.b $, because $E[Y_{l+2}, M_j]=\phi$ by definition of $G$, while in Subcase $i.4.b $ we had to prove it.\\

\textbf{\underline{Case $ 6 $:}} \textit{$f=u$}. Clearly, $u$ gains only $v$ as a first out-neighbor and gains $y$ as a second out-neighbor. We prove that $u$ gains only $y$ as a second out-neighbor. But this is exactly same as Case $iii.6$.\\

\textbf{\underline{Case $ 7 $:}} \textit{$f=v$}. It is clear that $v$ gains no new first out-neighbor. We prove that it gains no new  second out-neighbor. Suppose that $v\ra a\ra b\ra v$ in $T'$. Then $(v,a)\in E(D)$ and $a\notin K\cup \{x,v,u\}$. Since $ v \ra y \ra u$  in $D$, we may assume that $b  \notin \{u,v, y\}$ and $a \neq y$.\\

\textit{Subcase a:} If $(a,b)\in E(D)$ or is a convenient orientation w.r.t. $D$, then $b\in N^{++}_D(f)$.\\

\textit{Subcase b:} If $(a,b)\in E(D)-E(D')$ and it is not convenient w.r.t. $D$, then there is $rs \rightarrow ab $ in $\Delta$, namely $s \rightarrow b$ and $ a \notin N^{++}_D(s)$. So either $rs, ab \in E(C)$ or $ rs, ab \in E[Y_j,M_j]$ for some $j$. If $ab \in E(C)$, then $(a,b)=(x,y)$ or $(a,b)=(z,u)$  and hence $b \in \{u, y\}$, a contradiction. Thus $rs, ab \in E[Y_j,M_j]$ for some $j$.  Since $ a\notin K$, then $ a \in M_j$ and thus $ s \in M_j$. So by definition of $G$, $fs$ is not a missing edge of $D$. But $(f,a)\in  E(D)$ and $ a \notin N^{++}_D(s)$, then  by Lemma \ref{snlemma} we get $b\in N^{++}_D(f)$.\\

\textit{Subcase $c$}: If $(a,b)\in E(T)-E(D')$, then $b\in K$ and $(a,b)$ is a convenient orientation w.r.t. $D'$. Since $(f,a)\in  E(D)$ and so in $D'$, there exists $v'$ such that $f\ra v'\ra b$ in $D'$. Since $(v',b)\in E(D')$ and $b\in K$, then $v'\notin C \cup k$. Since $v'\notin C$ and $(f,v')\in E(D')$, then $(f,v')\in E(D)$. But this is already treated in Subcase $iv.7.a$ and Subcase $iv.7.b$.\\

\textbf{\underline{Case $ 8$:}} \textit{$f=x$}. Then $x$ gains only $y$ as a first out-neighbor and  gains $u$ as a second out-neighbor. We prove that it gains only $u$ as a new second out-neighbor.\\

\textit{Subcase a:} Suppose that $x\ra y\ra b\ra x$ in $T'$ with $b\neq u$. Then $ (y,b) \in E(T)$. Since $x\ra z$ and $v\ra y$ in $D$, then we may assume $b\notin C$ and hence $(y,b) \notin E(D')-E(D)$. If $ (y,b) \in E(D)$,  then $b$ is a whole vertex or  $ b \in S$. Thus by the losing relation $xy\ra zu$, we get $b\in N^{++}_D(x)$. Else if $ (y,b) \in E(T)-E(D')$, then  $b\in K$ and  $ (y,b)$ is a convenient orientation w.r.t. $D'$. So there is $v'$ such that $x\ra v'\ra b$ in $D'$. Since $(v',b)\in E(D')$ and $b\in K$, then $v'\notin C \cup K$. Since $v'\notin C$ and $(f,v')\in E(D')$, then $(f,v')\in E(D)$. As usual we can prove that $b\in N^{++}_D(x)$ in $D$. \\

\textit{Subcase b:} Suppose that $x\ra a\ra b\ra x$ in $T'$ with $a\neq y$ and $b\neq u$. Then $(x,a)\in D$ and  thus $a\notin K$. Since  $x\ra z$ in $D$ and $x\ra y$ in $T$, then $b\notin \{x,y,u,z\}$. We proceed exactly as in Case $iv.7$.\\

\textbf{\underline{Case $ 9$:}} \textit{$f=y$}. Clearly, $y$ gains no new  first out-neighbor. We prove that  it gains no new second out-neighbor. Suppose that $y\ra a\ra b\ra y$ in $T'$. Then $(y,a)\in E(D)$ and $a\notin K$. Since $y\ra u\ra x$ in $D$, then we may assume $b\notin \{x, y,u\}$. We continue  exactly as in  Case $iv.7$.\\

\textbf{\underline{Case $ 10$:}} \textit{$f=z$}. It is clear that $z$ gains no new vertex as a first out-neighbor. We prove that it gains no new vertex as a second out-neighbor. Suppose that $z\ra a\ra b\ra z$ in $T'$. Then $(z,a)\in E(D)$ and $a\notin K$. Since $x\ra z$ in $D$,  we may assume that  $a\notin K\cup \{x,z\}$.\\

\textit{Subcase a:} If $(a,b)\in E(D)$ or it is a convenient orientation w.r.t. $D$, then $b\in N^{++}_D(z)$.\\

\textit{Subcase b:} If $(a,b)\in E(D)-E(D')$ and it is not convenient w.r.t. $D$, then there is $rs \rightarrow ab $ in $\Delta$, namely $s \rightarrow b$ and $ a \notin N^{++}_D(s)$. So either $rs, ab \in E(C)$ or $ rs, ab \in E[Y_j,M_j]$ for some $j$. If $ab \in E(C)$, then $(a,b)=(x,y)$ or $(a,b)=(z,u)$  and hence $a \in \{x, z\}$, a contradiction. Thus $rs, ab \in E[Y_j,M_j]$ for some $j$.  Since $ a\notin K$, then $ a \in M_j$ and thus $ s \in M_j$. So by definition of $G$, $fs$ is not a missing edge of $D$. But $(f,a)\in  E(D)$ and $ a \notin N^{++}_D(s)$, then  by Lemma \ref{snlemma} we get $b\in N^{++}_D(f)$.\\

\textit{Subcase $c$}: If $(a,b)\in E(T)-E(D')$, then $b\in K$ and $(a,b)$ is a convenient orientation w.r.t. $D'$. Since $(f,a)\in  E(D)$ and so in $D'$, there exists $v'$ such that $f\ra v'\ra b$ in $D'$. Since $(v',b)\in E(D')$ and $b\in K$, then $v'\notin C \cup k$. Since $v'\notin C$ and $(f,v')\in E(D')$, then $(f,v')\in E(D)$. But this is already treated in Subcase $iv.10.a$ and Subcase $iv.10.b$.\\

\textbf{\underline{Case $  11$:}} \textit{$f\in V(G)-(Y\cup M \cup C) = A \cup (X-X_1)= A \cup (X-Y_1) $}.  Exactly same as Case $ i.6 $, with only one difference in Subcase $ b $. The difference is that in Subcase $iv.11.b$ there are two possibilities for the edges $rs, ab$: Either  $rs, ab \in E[Y_j, M_j]$ for some $j$,  or $rs, ab \in E(C)$. The first case is treated in Subcase $i.6.b$. However, for the second case, $f$ must belong to $A$ since otherwise $fa$ is a missing edge because $G[X \cup C]$  is a complete split graph, which contradicts the fact that $(f,a) \in E(D)$. Thus $fs$ is not missing edge of $D$ because $E[A, C]=\phi$ by definition of $G$. Since $f\ra a$ in $D$ and $a \notin N^{++}_D(s)$, then by Lemma \ref{snlemma} we get $b \in N^{++}_D(f)$.\\

Therefore, $f$ satisfies the SNP in $D$.\\

\noindent \textbf{\underline{Case $ v $:}} \textit{Suppose that $\Delta[E(C)]$ has exactly three arcs, say $uv\ra xy\ra zu\ra vx$.}\\

\noindent Then  $u\ra x\ra z\ra v\ra y $ in $D$ and $uv, yz$ are good missing edges. Assume without loss of generality that $(u,v)$ is a convenient orientation w.r.t. $D$. Add the arcs $(u,v)$, $(x,y)$, $(z,u)$ and $(v,x)$ to $D$. Assign to  $yz$ a convenient orientation and add it to $D$. The obtained oriented graph $D'$ is missing  $G'=G-(\cup E[Y_j, M_j] \cup E(C))$ which is a threshold graph.  So all the missing edges of $D'$ are good. We assign to them a convenient orientation and we add them to get a tournament $T$. Let $L$ be a local median order of $T$ and let $f$ denote its feed vertex. Reorient all the missing edges incident to $f$ towards $f$ except  those whose out-degree in $\Delta$ is not zero. The same order $L$ is again a local median order of the obtained tournament $T'$ and $f$ has the SNP in $T'$. We will prove that $f$  has the SNP in $D$. We have the following cases.\\

\textbf{\underline{Case $ 1 $:}} \textit{$f$ is a whole vertex.} This is the same as Case $i.1$.\\

\textbf{\underline{Case $ 2 $:}} \textit{$\exists$ $ 1\leq t\leq l$ such that $f\in M_t$.} Exactly same as Case $i.2$, with only one difference in the unsteady subcases, that is, in the subcases where $f \ra a \ra b$ with $(f,a) \in E(D)$ and $(a,b) \in E(D')-E(D)$ and it is not convenient w.r.t. $D$.  As usual, since $(a,b)$ is not convenient w.r.t. $D$, then there is  \lose. The difference is that in the unsteady subcases of Case $ v.2 $  either  $rs, ab \in E[Y_j, M_j]$ for some $j$ or $rs, ab \in E(C)$. If $rs, ab \in E[Y_j, M_j]$ for some $j$, we proceed exactly in the same way as in the unsteady subcases of Case $i.2$.  Else if $rs, ab \in E(C)$ (the possible cases are: $(r,s)= (u,v)$ and $ (a,b)=(x,y)$ or  $(r,s)= (x,y)$ and $ (a,b)=(z,u)$ or $(r,s)= (z,u)$ and $ (a,b)=(v,x)$), then $fs$ is not a missing edge because $E[M, C] = \phi$.  Since $f\ra a$ in $D$ and $a \notin N^{++}_D(s)$, then by Lemma \ref{snlemma} we get $b \in N^{++}_D(f)$.\\

\textbf{\underline{Case $ 3 $:}} \textit{$\exists$ $ 1\leq t\leq l$ such that $f\in Y_t$ and $M_t\neq \phi$}. Exactly same as Case $i.3$, with only one difference in the subcases when $f \ra a \ra b$ with $(f,a) \in E(D)$ and $(a,b) \in E(D')-E(D)$ and it is not convenient w.r.t. $D$. As usual, since $(a,b)$ is not convenient w.r.t. $D$, then there is  \lose. The difference is that in the unsteady subcases of Case $ v.3 $ there are two cases to be consider: Either  $rs, ab \in E[Y_j, M_j]$ for some $j$, or $rs, ab \in E(C)$. If $rs, ab \in E[Y_j, M_j]$ for some $j$, we proceed exactly in the same way as in the unsteady subcases of Case $i.3$.  Else if $rs, ab \in E(C)$, then $fa$ is  a missing edge because $G[Y \cup C]$ is a complete split graph, a contradiction. Thus this case does not exist.\\

\textbf{\underline{Case $ 4 $:}} \textit{$\exists$ $ 1\leq t\leq l+1$ such that $f\in Y_t$ such that $M_t=\phi$}. Exactly same as Case $ i.4 $, with only one difference in Subcase $ b $. The difference is that in Subcase $v.4.b$ there are two possibilities for the edges $rs, ab$: Either  $rs, ab \in E[Y_j, M_j]$ for some $j$,  or $rs, ab \in E(C)$. The first case is treated in Subcase $i.4.b$. However, the second case does not exist since otherwise  $fa$ is  a missing edge because $G[Y \cup C]$ is a complete split graph, which contradicts the fact that $(f,a) \in E(D)$. This means that  this case does not exist. \\

\textbf{\underline{Case $ 5 $:}} \textit{$f\in Y_{l+2}$}. Exactly same as Case $ i.4 $, with two differences in Subcase $ b $. The first difference is that in Subcase $v.5.b$ there are two possibilities for the edges $rs, ab$: Either  $rs, ab \in E[Y_j, M_j]$ for some $j$,  or $rs, ab \in E(C)$. The first case is already treated in Subcase $i.4.b$. However, the second case does not exist since otherwise  $fa$ is  a missing edge because $G[Y \cup C]$ is a complete split graph, which contradicts the fact that $(f,a) \in E(D)$. The second  difference is that   $fs$ is not a missing edge in Subcase $ v.5.b $, because $E[Y_{l+2}, M_j]=\phi$ by definition of $G$, while in Subcase $i.4.b $ we had to prove it.\\

\textbf{\underline{Case $ 6 $:}} \textit{$f=u$}. Exactly same as Case $iv.6$.\\

\textbf{\underline{Case $ 7 $:}} \textit{$f\in \{x,z\}$}. Similar to the case $f=u$, that is to Case $v.6$.\\

\textbf{\underline{Case $ 8$:}} \textit{$f=y$}.   Exactly same as Case $ iii.9 $ with only one difference in Subcase $ b $. The  difference is that in Subcase $v.8.b$ when  $ab \in E(C)$  then $(a,b)$ can be either $(x,y)$ or $(z,u)$ or $(v,x)$ and so $a \in \{x,z,v\}$, a contradiction.\\

\textbf{\underline{Case $ 9$:}} \textit{$f=v$}. Similar  to the case $f=y$, that is to Case $v.8$.\\

\textbf{\underline{Case $  10$:}} \textit{$f\in V(G)-(Y\cup M \cup C) = A \cup (X-X_1)= A \cup (X-Y_1) $}.  Exactly same as Case $ i.6 $, with only one difference in Subcase $ b $. The difference is that in Subcase $v.10.b$ there are two possibilities for the edges $rs, ab$: Either  $rs, ab \in E[Y_j, M_j]$ for some $j$,  or $rs, ab \in E(C)$. The first case is treated in Subcase $i.6.b$. However, for the second case, $f$ must belong to $A$ since otherwise $fa$ is a missing edge because $G[X \cup C]$  is a complete split graph, which contradicts the fact that $(f,a) \in E(D)$. Thus $fs$ is not missing edge of $D$ because $E[A, C]=\phi$ by definition of $G$. Since $f\ra a$ in $D$ and $a \notin N^{++}_D(s)$, then by Lemma \ref{snlemma} we get $b \in N^{++}_D(f)$.\\

Therefore, in view of all above observations, $f$ satisfies the SNP in $D$.\\

\noindent \textbf{\underline{Case $ vi $:}} \textit{Suppose that $\Delta[E(C)]$ has exactly  three arcs, say $uv\ra xy\ra zu$ and  $xv\ra zy$}.\\

\noindent Then $u\ra x\ra z\ra v\ra y\ra u$ in $D$ and $uv$ and $xv$ are good missing edges. Assume without loss of generality that $(u,v)$ is a convenient orientation  of $uv$. Add $(u,v)$, $(x,y)$ and $(z,u)$ to $D$. If $(x,v)$ is a convenient orientation, then add $(x,v)$ and $(z,y)$ to $D$, otherwise add $(v,x)$ and $(y,z)$ to $D$.  The obtained oriented graph $D'$ is missing  $G'=G-(\cup E[Y_j, M_j] \cup E(C))$ which is a threshold graph.  So all the missing edges of $D'$ are good. We assign to them a convenient orientation and we add them to get a tournament $T$. Let $L$ be a local median order of $T$ and let $f$ denote its feed vertex. Reorient all the missing edges incident to $f$ towards $f$ except  those whose out-degree in $\Delta$ is not zero. The same order $L$ is again a local median order of the obtained tournament $T'$ and $f$ has the SNP in $T'$. We will prove that $f$  has the SNP in $D$. We have the following cases.\\

\textbf{\underline{Case $ 1 $:}} \textit{$f$ is a whole vertex.} This is the same as Case $i.1$.\\

\textbf{\underline{Case $ 2 $:}} \textit{$\exists$ $ 1\leq t\leq l$ such that $f\in M_t$.} Exactly same as Case $i.2$, with only one difference in the unsteady subcases, that is, in the subcases where $f \ra a \ra b$ with $(f,a) \in E(D)$ and $(a,b) \in E(D')-E(D)$ and it is not convenient w.r.t. $D$.  As usual, since $(a,b)$ is not convenient w.r.t. $D$, then there is  \lose. The difference is that in the unsteady subcases of Case $ vi.2 $  either  $rs, ab \in E[Y_j, M_j]$ for some $j$ or $rs, ab \in E(C)$. If $rs, ab \in E[Y_j, M_j]$ for some $j$, we proceed exactly in the same way as in the unsteady subcases of Case $i.2$.  Else if $rs, ab \in E(C)$ (the possible cases are: $(r,s)= (u,v)$ and $ (a,b)=(x,y)$, $(r,s)= (x,y)$ and $ (a,b)=(z,u)$, $(r,s)= (v,x)$ and $ (a,b)=(y,z)$ if $ (v,x) $ is a convenient orientation of $vx$ or  $(r,s)= (x,v)$ and $(a,b)=(z,y)$ if $ (x,v) $ is a convenient orientation of $vx$), then $fs$ is not a missing edge because $E[M, C] = \phi$.  Since $f\ra a$ in $D$ and $a \notin N^{++}_D(s)$, then by Lemma \ref{snlemma} we get $b \in N^{++}_D(f)$.\\

\textbf{\underline{Case $ 3 $:}} \textit{$\exists$ $ 1\leq t\leq l$ such that $f\in Y_t$ and $M_t\neq \phi$}. Exactly same as Case $i.3$, with only one difference in the subcases when $f \ra a \ra b$ with $(f,a) \in E(D)$ and $(a,b) \in E(D')-E(D)$ and it is not convenient w.r.t. $D$. As usual, since $(a,b)$ is not convenient w.r.t. $D$, then there is  \lose. The difference is that in the unsteady subcases of Case $ vi.3 $ there are two cases to be consider: Either  $rs, ab \in E[Y_j, M_j]$ for some $j$, or $rs, ab \in E(C)$. If $rs, ab \in E[Y_j, M_j]$ for some $j$, we proceed exactly in the same way as in the unsteady subcases of Case $i.3$.  Else if $rs, ab \in E(C)$, then $fa$ is  a missing edge because $G[Y \cup C]$ is a complete split graph, a contradiction. Thus this case does not exist.\\

\textbf{\underline{Case $ 4 $:}} \textit{$\exists$ $ 1\leq t\leq l+1$ such that $f\in Y_t$ such that $M_t=\phi$}. Exactly same as Case $ i.4 $, with only one difference in Subcase $ b $. The difference is that in Subcase $vi.4.b$ there are two possibilities for the edges $rs, ab$: Either  $rs, ab \in E[Y_j, M_j]$ for some $j$,  or $rs, ab \in E(C)$. The first case is treated in Subcase $i.4.b$. However, the second case does not exist since otherwise  $fa$ is  a missing edge because $G[Y \cup C]$ is a complete split graph, which contradicts the fact that $(f,a) \in E(D)$. This means that  this case does not exist. \\

\textbf{\underline{Case $ 5 $:}} \textit{$f\in Y_{l+2}$}. Exactly same as Case $ i.4 $, with two differences in Subcase $ b $. The first difference is that in Subcase $vi.5.b$ there are two possibilities for the edges $rs, ab$: Either  $rs, ab \in E[Y_j, M_j]$ for some $j$,  or $rs, ab \in E(C)$. The first case is already treated in Subcase $i.4.b$. However, the second case does not exist since otherwise  $fa$ is  a missing edge because $G[Y \cup C]$ is a complete split graph, which contradicts the fact that $(f,a) \in E(D)$. The second  difference is that   $fs$ is not a missing edge in Subcase $ vi.5.b $, because $E[Y_{l+2}, M_j]=\phi$ by definition of $G$, while in Subcase $ i.4.b $ we had to prove it.\\

\textbf{\underline{Case $ 6$:}} \textit{$f=u$}.  Exactly same as Case $iv.6$.\\

\textbf{\underline{Case $ 7$:}} \textit{$f=y$}. Exactly same as Case $ iii.9 $ with only one difference in Subcase $ b $. The  difference is that in Subcase $vi.7.b$ if  $ab \in E(C)$  then $(a,b)$ can be either $(x,y)$,  $(z,u)$  $(y,z)$ or $(z,y)$ and so $a \in \{x,y,z\}$, a contradiction.\\

\textbf{\underline{Case $ 8$:}} \textit{$f=z$}. Exactly same as Case $vi.7$ with difference that $yz$ and $uz$ are  reoriented so that $(y,z)\in E(T')$ and $(u,z) \in E(T')$, respectively.\\

\textbf{\underline{Case $ 9$:}} \textit{$f=v$}. Exactly same as Case $iii.7$, with only one difference: In Subcase $vi.9.2.b$ if  $ab \in E(C)$  then $(a,b)$ can be either $(x,y)$,  $(z,u)$  or $(z,y)$ and so $a \in \{x,z\}$, a contradiction because $z \ra v$ in $D$ and $x \ra v $ in $D'$ while $v \ra a$ in $D$.\\

\textbf{\underline{Case $ 10$:}} \textit{$f=x$}. Here we consider two main cases: \\

\textit{\underline{Case $ 10.1$:} Assume  $(x,v)\in E(D')$.} Clearly, $x$ gains only $v$ and $y$ as new first out-neighbors. However, $x$ loses $v$ as a second out-neighbor and gains $u$ as a new second out-neighbor. We prove that $x$ gains only $u$ as a new second out-neighbor.\\

\textit{Subcase a:} Suppose that $x\ra a \ra b \ra x $ in $T'$, with $a\neq y$, $a\neq v$ and $b\neq u$. Then $(a,b ) \in E(T)$, $(x,a)\in E(D)$  and thus $a\notin K$. Since $x\ra z\ra v$ in $D$ and $x\ra y$ in $T$, then we may assume that $b\notin C$.\\

\textit{Subcase $a.1$}: If $(a,b)\in E(D)$, $b\in N^{++}_D(x)$.\\

\textit{Subcase $a.2$}: If $(a,b)\in E(D')-E(D)$, then either $ab \in E(C)$ or $ ab \in E[Y_j,M_j]$ for some $j$. If $ab \in E(C)$, then $b \in C$, a contradiction. Thus $ ab \in E[Y_j,M_j]$ for some $j$. It follows that either $(a,b)$ is a convenient orientation w.r.t. $D$ and hence $b\in N^{++}_D(x)$ or there is $rs \rightarrow ab $ in $\Delta$, namely $s \rightarrow b$ and $ a \notin N^{++}_D(s)$. Since $ a\notin K$, then $ a \in M_j$ and thus $ s \in M_j$. So by definition of $G$, $xs$ is not a missing edge of $D$. But $(x,a)\in  E(D)$ and $ a \notin N^{++}_D(s)$, then  by Lemma \ref{snlemma} we get $b\in N^{++}_D(x)$.\\

\textit{Subcase $a.3$}: If $(a,b)\in E(T)-E(D')$, then $b\in K$ and $(a,b)$ is a convenient orientation w.r.t. $D'$. Since $x \ra a$ in  $D$ and so in $D'$, there exists $v'$ such that $x\ra v'\ra b$ in $D'$. Since $(v',b)\in E(D')$ and $b\in K$, then $v'\notin C$. Since $v'\notin C$ and $(x,v')\in E(D')$, then $(x,v')\in E(D)$. But this is already treated in Subcase $vi.10.1.a.1$ and Subcase $vi.10.1.a.2$.\\

\textit{Subcase $ b $:} Suppose that $x\ra y\ra b\ra x$ in $T'$, with $b\neq u$. Since $x\ra z\ra v$ in $D$, then we may assume that $b\notin C$ and hence $(y,b) \notin E(D')-E(D)$.\\

\textit{Subcase $b.1$}: If $(y,b)\in E(D)$, then either  $b \in S$ or $b$ is a whole vertex. Then by the losing relation $xy \ra zu$, we get $b\in N^{++}_D(x)$.\\

\textit{Subcase $b.2$}: If $(y,b)\in E(T)-E(D')$, then $b\in K$ and $(y,b)$ is a convenient orientation w.r.t. $D'$. But this is exactly the same as Subcase $vi.10.1.a.1$ and Subcase $vi.10.1.a.2$.\\

\textit{Subcase $ c $:} Suppose that $x\ra v\ra b\ra x$ in $T'$, with $b\neq u$. Since $x\ra z\ra v$ in $D$ and $x\ra y$ in $T$, then we may assume that $b\notin C$.  We proceed similarly to  Subcase $vi.10.1.b$ by replacing the losing relation $xy \ra zu$ in Subcase $vi.10.1.b$ by the losing relation $xv \ra zy$ in Subcase $vi.10.1.c.1$. \\

\textit{\underline{Case $ 10.2$:} Assume  $(v,x)\in E(D')$.}  Clearly, $x$ gains only $y$ as a new first out-neighbor and $u$ as a new second out-neighbor. We prove it gains only $u$ as a new second out-neighbor.\\

\textit{Subcase a:} Suppose that $x\ra y\ra b\ra x$ in $T'$ with $b\neq u$. Then $ (y,b) \in E(T)$. Since $x\ra z \ra V$ in $D$, then we may assume $b\notin C$ and hence $(y,b) \notin E(D')-E(D)$. This is exactly as Subcase $vi.10.1.b$.	 \\

\textit{Subcase b:} Suppose that $x\ra a\ra b\ra x$ in $T'$ with $a\neq y$ and $b\neq u$. Then $(a,b ) \in E(T)$, $(x,a)\in E(D)$ and  thus $a\notin K$. Since  $x\ra z \ra v$ in $D$ and $x\ra y$ in $T$, then $b\notin C$. We proceed exactly as in Case $vi.10.1.a$.\\

\textbf{\underline{Case $  11$:}} \textit{$f\in V(G)-(Y\cup M \cup C) = A \cup (X-X_1)= A \cup (X-Y_1) $}.  Exactly same as Case $ i.6 $, with only one difference in Subcase $ b $. The difference is that in Subcase $vi.11.b$ there are two possibilities for the edges $rs, ab$: Either  $rs, ab \in E[Y_j, M_j]$ for some $j$,  or $rs, ab \in E(C)$. The first case is treated in Subcase $i.6.b$. However, for the second case, $f$ must belong to $A$ since otherwise $fa$ is a missing edge because $G[X \cup C]$  is a complete split graph, which contradicts the fact that $(f,a) \in E(D)$. Thus $fs$ is not missing edge of $D$ because $E[A, C]=\phi$ by definition of $G$. Since $f\ra a$ in $D$ and $a \notin N^{++}_D(s)$, then by Lemma \ref{snlemma} we get $b \in N^{++}_D(f)$.\\

Therefore,  all what precede prove that $f$ has the SNP in $D$. This completes the proof.

\end{proof}

As immediate consequences of the previous theorem, we may conclude the following:

\begin{corollary}
	Every oriented graph missing a generalized comb satisfies the SNC.
\end{corollary}

\begin{corollary} (Ghazal \cite{contrib} )
	Every oriented graph missing a comb satisfies the  SNC.
\end{corollary}

\begin{corollary} (Ghazal \cite{a})\label{SNCgs}
	Every oriented graph missing a threshold graph satisfies the SNC.\\
\end{corollary}

Since  threshold graphs, $C_5$, generalized combs and $\{C_{4}, \overline{C_{4}}, S_{3},$ chair and co-chair$\}$-free graphs are in $ \mathcal{F}(\vec{\mathcal{P}})$ and any   oriented graph missing one of the  graphs  mentioned before satisfies the SNC, we end this article by wondering the following:

\begin{problem}
	Does every oriented graph missing a graph in $\mathcal{F}(\vec{\mathcal{P}})$ satisfies SNC?
\end{problem}

\end{document}